\documentclass[reqno, 12pt]{amsart}
  \makeatletter
  \let\origsection=\section 
  \def\section{\@ifstar{\origsection*}{\mysection}}
  \def\mysection{\@startsection{section}{1}\z@{.7\linespacing\@plus\linespacing}{.5\linespacing}{\normalfont\scshape\centering\S}}
  \makeatother

  \usepackage{amsmath,amssymb,amsthm}
  \usepackage{mathrsfs}
  \usepackage{dsfont}
  \usepackage{mathabx}\changenotsign
  \usepackage[utf8]{inputenc}
  \usepackage{microtype}
  \usepackage{verbatim}
  \usepackage{todonotes}
  \usepackage{enumerate}
  \usepackage{cancel}
  \usepackage{enumitem}

  \usepackage{xcolor}
  \usepackage[backref=page]{hyperref}
  \hypersetup{%
      colorlinks,
      linkcolor={red!60!black},
      citecolor={green!60!black},
      urlcolor={blue!60!black}
  }

\let\phi=\varphi
\def\red{\text{\rm red}}
\def\blue{\text{\rm blue}}
\def\green{\text{\rm green}}

  \usepackage{bookmark}
  \usepackage{booktabs}

  \usepackage[abbrev,msc-links,backrefs]{amsrefs}
  \usepackage{doi}

  \renewcommand{\PrintDOI}[1]{\doi{#1}}

  \usepackage[T1]{fontenc}
  \usepackage{lmodern}

  \usepackage[english]{babel}

  \usepackage{tikz}
  \usetikzlibrary{matrix}

  \usepackage{geometry}
  \let\polishlcross=\l
  \def\l{\ifmmode\ell\else\polishlcross\fi}

  \renewcommand{\setminus}{\smallsetminus}

  \makeatletter
  \def\moverlay{\mathpalette\mov@rlay}
  \def\mov@rlay#1#2{\leavevmode\vtop{%
     \baselineskip\z@skip \lineskiplimit-\maxdimen
     \ialign{\hfil$\m@th#1##$\hfil\cr#2\crcr}}}
  \newcommand{\charfusion}[3][\mathord]{
      #1{\ifx#1\mathop\vphantom{#2}\fi
          \mathpalette\mov@rlay{#2\cr#3}
        }
      \ifx#1\mathop\expandafter\displaylimits\fi}
  \makeatother

  \newtheorem{theorem}{Theorem}
  \newtheorem{lemma}[theorem]{Lemma}

  \newtheorem*{remark*}{Remark}

  \newtheorem{fact}[theorem]{Fact}

\begin{document}

\title[Counting Gallai $3$-colorings of complete graphs]{Counting Gallai $3$-colorings of complete graphs}

\author[J.~de~O.~Bastos]{Josefran de Oliveira Bastos}
\address{Engenharia da Computação, Universidade Federal do Ceará, Ceará, Brazil}
\email{josefran@ufc.br}

\author[F.~S.~Benevides]{Fabrício Siqueira Benevides}
\address{Departamento de Matemática, Universidade Federal do Ceará, Ceará, Brazil}
\email{fabricio@mat.ufc.br}

\author[G.~O.~Mota]{Guilherme Oliveira Mota}
\address{Centro de Matem\'atica, Computa\c c\~ao e Cogni\c c\~ao, Universidade Federal do ABC, Santo Andr\'e, Brazil}
\email{g.mota@ufabc.edu.br}

\author[I.~Sau]{Ignasi Sau}
\address{CNRS, LIRMM, Universit\'e de Montpellier, Montpellier, France}
\email{ignasi.sau@lirmm.fr}

\thanks{The first and second author were supported by CAPES Probral (Proc. 88887.143992/2017-00); the second author by CNPQ (Proc. 310512/2015-8 and Proc. 425297/2016-0) and FUNCAP; the third author by FAPESP (Proc. 2013/11431-2); and the fourth author was by projects DEMOGRAPH (ANR-16-CE40-0028) and ESIGMA (ANR-17-CE40-0028).}

\begin{abstract}
An edge coloring of the $n$-vertex complete graph $K_n$ is a \emph{Gallai coloring} if it does not contain any rainbow triangle, that is, a triangle whose edges are colored with three distinct colors.
We prove that the number of Gallai colorings of $K_n$ with at most three colors is at most $7(n+1)\,2^{\binom{n}{2}}$, which improves the best known upper bound of $\frac{3}{2}(n-1)!\cdot 2^{\binom{n-1}{2}}$ in [Discrete Mathematics, 2017].
\end{abstract}


\maketitle

\section{Introduction}

  An edge coloring of the complete graph $K_n$ is a \emph{Gallai coloring} if it contains no \mbox{\emph{rainbow}~$K_3$}, that is, a copy of $K_3$ in which all edges have different colors. A $t$-coloring is a edge coloring that uses at most $t$ colors.
      The term \emph{Gallai coloring} was used by Gyárfás and Simonyi in~\cite{GySi2004edge}, but those colorings have also been studied under the name \emph{Gallai partitions} by Körner, Simonyi and Tuza in~\cite{korner1992perfect}. See also \cite{gyarfas2010gallai} for a generalization to non-complete graphs and \cite{chua2013gallai} for hypergraphs. The nomenclature is due to a close relation to a result in the influential Gallai's original paper \cite{gallai1967transitiv} -- translated to English and endowed by comments in \cite{gallai1967-translation}. The above mentioned papers are mostly concerned with structural and Ramsey-type results about Gallai colorings. For example, in \cite{GySi2004edge} it was proved that any Gallai coloring can be obtained by substituting complete graphs with Gallai colorings into vertices of $2$-colored complete graphs (\cite[Theorem A]{GySi2004edge}), and that any Gallai coloring contains a monochromatic spanning tree (\cite[Theorem 2.2]{GySi2004edge}).

We are interested in the problem of counting the number of Gallai colorings of $K_n$ with a fixed set of colors, and focus here on the case where we use at most three colors.
Here, we consider that the vertices of $K_n$ are labeled. This problem has been investigated in the recent literature by other authors (under the more descriptive name \emph{rainbow triangle-free colorings}). Actually, in a more general setting, in the past years there has been a growth on the number of results about counting the number of structures that do not contain a particular kind of substructure, due to the recent development of modern and classic methods such as the Containers Method~\cites{balogh2015independent,saxton2015hypergraph}, Regularity Method~\cites{AlBoHaKoPe18+,Ko97, komlos1996szemeredi, KoShSiSz02}, and the Entropy Compression Method~\cites{AlPrSa15,EsPa13}: for example, counting sum-free sets in Abelian groups \cites{alon2014counting, balogh2015number}, graphs without given subgraphs \cites{balogh2011number,balogh2014number,morris2016number,kohayakawa1998extremal}, sets of integers with no $k$-term arithmetic progression \cite{balogh2016number}, and $B_h$-sets \cite{dellamonica2016number}, to cite only a few.

In \cite{benevides2017edge}, Benevides, Hoppen and Sampaio, motivated by a question of Erd\H{o}s and Rothschild (see, e.g.,~\cite{erdos1974some}),  studied the general problem of counting the number of colorings of a graph that avoid a subgraph colored with a given pattern (see also~\cites{alon2004number,hoppen2015rainbow, hoppen2017graphs,pikhurko2017erdHos}).
The problem of computing exactly the number of Gallai colorings of $K_n$, which we denote by $c(n)$, appears to be hard even when we restrict the number of colors to be at most three.
Since the vertices of $K_n$ are labeled, we have the trivial lower bound $c(n) \ge 3\cdot 2^{\binom{n}{2}}-3$, given by the colorings that use only one or two of the three colors.
To the best of our knowledge, the best known upper bound for $c(n)$ is $\frac{3}{2}(n-1)!\cdot 2^{\binom{n-1}{2}}$~(see~\cite{benevides2017edge}).
We remark that a previous work using entropy, graph limits and the Container Method, lead to a general result that in turn implies a weaker upper bound of the form $2^{(1+o(1))\binom{n}{2}}$ (see \cite{falgas2016multicolour}).

  Our main result (see Theorem~\ref{thm:main_result} below) improves the best known upper bound on $c(n)$.
  While most recent related results only work for sufficiently large structures, our result holds for every value of $n \ge 2$.
  Furthermore, we provide an elementary proof of such result, which is also relatively short.

  \begin{theorem}~\label{thm:main_result}
    For all $n \ge 2$ we have
    $$
    c(n) \leq 7(n+1)\,2^{\binom{n}{2}}.
    $$
  \end{theorem}
  

 Even though our proof is completely mathematical, we also used computer search to calculate the exact number of Gallai colorings of small complete graphs (see~Table~\ref{table:cn} in the appendix) and to list all Gallai colorings of $K_n$ for $n$ up to $5$.
 This provided us an insight on how to use a simple induction to estimate $c(n)$ for large values of $n$ and how organize some case analysis for small values of $n$.
 We believe our strategy could also give good bounds when more colors are allowed, or even improve the above bound with a finer analysis (but probably leading to a much longer proof than we would like to show here).

  Regardless of the nomenclature, it is also worth mentioning that Gallai colorings appear naturally in other fields, such as Information Theory \cite{korner2000graph} (what motivates the use of entropy for solving this kind of problems), and the perfect graph theorem \cite{cameron1986note}.

In Sections~\ref{sec:2col} and~\ref{sec:3col} we estimate the number of extensions of colorings of the complete graphs that use, respectively, exactly two colors and exactly three colors. Most proofs in those sections are done by induction. The base cases are usually tedious, so we moved some of them to the appendix (Section~\ref{sec:appendix}). For those cases we only need to check some particular colorings of $K_4$, $K_5$ and $K_6$, but we also did a extensive computer search for all Gallai colorings of $K_n$ for $n\le 8$, and computer checked our results up to this value. The appendix also contains a table with the exact values of $c(n)$ for $n\ge 8$. In Section~\ref{sec:proof-maintheorem} we give a proof of Theorem~\ref{thm:main_result}.

\section{Counting the maximum number of extensions}
\label{sec:counting-extensions}

Let $\Phi_n$ be the set of all Gallai colorings of $K_n$ that use colors red, green, and blue. So, each element of $\Phi_n$ is a function $\phi : E(K_n) \to \{\red, \green, \blue\}$, and $c(n) = |\Phi_n|$. For a fixed $\phi \in \Phi_n$, we denote by $w(\phi)$ the number of ways to extend $\phi$ to a Gallai coloring of the complete graph $K_{n+1}$. We think of $K_{n+1}$ as obtained from $K_n$ by adding a new vertex~$u$, and $w(\phi)$ as the number of ways to color all edges incident to $u$ without creating a rainbow triangle given that $E(K_n)$ had already been colored as in $\phi$. We start with the following trivial fact that calculates the number of extensions of a monochromatic coloring.

\begin{fact}\label{claim:monochromatic_extensions}
  Let $n \geq 2$ be an integer and let $\phi \in \Phi_n$ be a monochromatic coloring of the edges of $K_n$.
  Then,
  $$
  w(\phi) = 2^{n+1} - 1.
  $$
\end{fact}
\begin{proof}
Without loss of generality assume that all edges of $K_n$ are colored blue. Let $\{u\} = V(K_{n+1}) \setminus V(K_n)$. Notice that in any extension we cannot use colors red and green on two different edges between $u$ and $K_n$. So, all extensions either use only colors blue and red, or only blue and green. As the extension in which all edges are blue is counted in both cases, there are $2^n + 2^n - 1$ extensions, as claimed.
\end{proof}

An straightforward approach to estimated $c(n)$ is to use bounds on the parameter $w(\cdot)$ to bound $c(\cdot)$, as we have the trivial relation $c(n) = \sum_{\phi \in \Phi_{n-1}}w(\phi)$. Unfortunately, computing $w(\phi)$ for each $\phi \in \Phi_{n-1}$ may be as hard as the original problem. Also, a trivial but tight general upper bound on $w(\phi)$, for $\phi\in \Phi_{n-1}$, leads to a very weak upper bound on $c(n)$. But we can partition $\Phi_{n-1}$ into classes in a way that we know how to estimate the maximum value of $w(\cdot)$ in each of those classes and use this fact to get a better upper bound on $c(n)$.
So, our strategy is to partition $\Phi_{n-1}$ into three classes: the monochromatic colorings, the colorings that use exactly two colors, and those that use exactly three colors. We denote those classes, $\Phi_n(1)$, $\Phi_n(2)$, and $\Phi_n(3)$ respectively, so that $\Phi_n = \Phi_n(1) \cup \Phi_n(2) \cup \Phi_n(3)$ is a disjoint union. We compute the maximum possible value of $w(\phi)$ for $\phi$ in each of those classes and determine for which colorings this maximum is achieved. A coloring that has the maximum number of extensions among colorings in its class will be called \emph{extremal}. The underlying reasoning for this method to work is that there is a large gap between the maximums of those classes. Before that, we state a general result about the (maximum) number of extensions of an extension.

 \begin{lemma}\label{lemma:2t+1}
  Let $\phi\in \Phi_n$ be a Gallai coloring. If $\phi'$ is an extension of $\phi$ to $E(K_{n+1})$, then
  $$
  w(\phi') \le 2w(\phi)+1.
  $$
 \end{lemma}

 \begin{proof}
  Let $K_n$ be the complete graph on $n$ vertices, $V$ be its vertex set, and $\phi \in \Phi_n$ be a Gallai coloring.
  Let $\phi'$ be any extension of $\phi$ to $E(K_{n+1})$ and $u \notin V$ be the new vertex (added to obtain $K_{n+1}$).
  To count the number of Gallai extensions of $\phi'$ to $E(K_{n+2})$, we will add a new vertex $x$ and all edges from $x$ to $V\cup\{u\}$. We first color the edges from $x$ to~$V$.
  If we let $t = w(\phi)$, there are $t$ colorings, say $\phi_1, \ldots, \phi_t$, of the edges from $x$ to~$V$. For each $i\in \{1, \ldots, t\}$ we let $m_i$ be number of ways we can color the edge $ux$ given that we have colored the edges from $x$ to $V$ as in $\phi_i$.
  Clearly, $m_i \in \{0, 1, 2, 3\}$ and $w(\phi') = \sum_{i=1}^{t}m_i$.

  Fix any $i$, with $1\le i \le t$. Recall that the edges from $u$ to $V$ are already colored (in $\phi'$). If there is any vertex $v \in V$ such that $\phi'(xv) \neq \phi_i(uv)$, then there is a forbidden color for $xu$ and $m_i \le 2$. Therefore, the only way to have $m_i =3$ is when the coloring $\phi_i$ is such that $\phi_i(xy) = \phi(uy)$ for every $y$ in $V$ (and for such coloring we have, indeed, $m_i = 3$). This implies that $\sum_{i=1}^{t}m_i\le 2t+1$.
 \end{proof}

We remark that when $\phi$ is a monochromatic coloring and $\phi'$ is its monochromatic extension, by Fact~\ref{claim:monochromatic_extensions}, we have $w(\phi) = 2^{n+1}-1$ and $w(\phi') = 2^{n+2}-1$. Therefore, $w(\phi') = 2w(\phi) + 1$, which implies that Lemma~\ref{lemma:2t+1} is best possible.

  In what follows we introduce a few definitions that play an important role in our proofs. Given a coloring of the edges of $K_n$ and a vertex $v\in K_n$, we say that $v$ is \emph{monochromatic} if all edges incident to $v$ have the same color. We also say that $v$ is a \emph{red vertex} (resp. \emph{blue vertex}, \emph{green vertex}) if all edges incident to $v$ are red (resp. blue, green). 
	
	Given a vertex $v \in V(K_n)$, we denote by $\phi^{\overline{v}}$ the restriction of the coloring $\phi$ to the complete graph $K_{n-1}$ obtained by removing $v$ from $K_n$. We will define some \emph{special} colorings of $E(K_n)$ and later we will prove that those are the unique extremal colorings in each of their classes (see Figure~\ref{fig:SpecialColorings}).

	First, any monochromatic coloring is considered special. Next, consider the case where $E(K_n)$ is colored with exactly two colors. We say that $\phi \in \Phi_n(2)$ is \emph{vertex-special} if there is a monochromatic vertex $v$, say of color $c_1$, and all edges not incident to $v$ have the same color, say $c_2$, where $c_1\neq c_2$. We say that $\phi$ is \emph{edge-special} if all edges have the same color with the exception of exactly one edge.

 	Now assume that $\phi \in \Phi_n(3)$. In this case, we say that $\phi$ is \emph{vertex-special} if there is a monochromatic vertex $v$ in a color $c \in\{\red, \blue, \green\}$, and $\phi^{\overline{v}}$ is an edge-special coloring with colors $\{\red, \blue,\green\}\setminus \{c\}$.
	Furthermore, we say that $\phi$ is \emph{edge-special} if there are two non-adjacent edges with different colors $c_1$ and $c_2$ in $\{\red, \blue,\green\}$  and all other edges are colored with color $\{\red, \blue,\green\}\setminus \{c_1,c_2\}$.
	
	Finally, we say that $\phi$ is \emph{special} to mean that $\phi$ is vertex-special or edge-special, and it is \emph{non-special} otherwise.
	See Figure~\ref{fig:SpecialColorings} for the non-monochromatic special $3$-colorings of $K_7$.

  \tikzset{edgeblue/.style={thick, color=blue}}
  \tikzset{edgered/.style={dashed, thick, color=red}}
  \tikzset{edgegreen/.style={dotted, line width=1mm, color=green}}

\begin{figure}[htb] 
\begin{center}
\newcommand{\coloringEdgeSpecial}[3]{%

	\begin{scope}[yshift=-0.2cm]
		\coordinate (p0) at (-1.2, 0.8);
		\coordinate (p1) at (-0.8, 0.3);
		\coordinate (p2) at (0, 0);
		\coordinate (p3) at (0.8, 0.3);
		\coordinate (p4) at (1.3, 0.8);
	\end{scope}
	\node[label=$v$, vertex] (v) at (-1,2.5){};
	\node[label=$w$, vertex] (w) at ( 1,2.5){};

	\draw[fill=gray!03] (0,0.4) ellipse (1.5cm and 1cm);
	\draw[#3] (v) -- (w);

	\draw[#1] (p0) -- (p1) -- (p2) -- (p3) -- (p4) -- (p0);
	\draw[#1] (p0) -- (p2) -- (p4) -- (p1) -- (p3) -- (p0);

	\draw[#2] (p0) -- (p4); 

	\draw[#1] (w) to [out=235, in=20] (p0);
	\draw[#1] (w) to [out=240, in=40] (p1);
	\draw[#1] (w) -- (p2);
	\draw[#1] (w) -- (p3);
	\draw[#1] (w) to [out=270, in=110] (p4);

	\draw[#1] (v) to [out=270, in=80] (p0);
	\draw[#1] (v) -- (p1);
	\draw[#1] (v) -- (p2);
	\draw[#1] (v) to [out=300, in=140] (p3);
	\draw[#1] (v) to [out=305, in=160] (p4);
}

\newcommand{\coloringVertexSpecial}[3]{%
	\begin{scope}[yshift=0.25cm]

	  \foreach \i in {0,...,3}{
	    \coordinate (p-\i) at (60*\i:1);	
	   }
	   \coordinate (p-4) at (220:0.7);	
	   \coordinate (p-5) at (320:0.7);	
	\end{scope}

	\node[label=$v$, vertex] (v) at (0,2.5){};

	\draw[fill=gray!03] (0,0.4) ellipse (1.5cm and 1cm);

	\foreach \i in {0,...,5}{
	\foreach \j in {1,...,5}{
	\draw[#1] (p-\i) -- (p-\j);
	}
	}

	\draw[#3] (p-4) -- (p-5);

    \draw[#2] (v) to [out=310, in=90] (p-0);
    \draw[#2] (v) to [out=300, in=90] (p-1);
    \draw[#2] (v) to [out=240, in=90] (p-2);
    \draw[#2] (v) to [out=230, in=90] (p-3);
    \draw[#2] (v) to [out=260, in=80] (p-4);
    \draw[#2] (v) to [out=280, in=100] (p-5);
}
\begin{tikzpicture}[text depth=0.25ex]
  \tikzset{vertex/.style={circle, fill=black!50, draw, inner sep=0pt}}

  \matrix [column sep=0.6cm,row sep=0.5cm]
  {
      \node (m-1-1) {};&
      \node[text width=7em] (m-1-2) {vertex-special};&
      \node[text width=6em] (m-1-3) {edge-special};
      \\
      \node[text width=4.9cm, anchor=south] (Case1) at (0,1) {using exactly two colors};&
      \coloringVertexSpecial{red}{blue, thick}{red} &
      \coloringEdgeSpecial{red}{red}{blue, line width=1.6pt} &
      \\
      \node[text width=4.9cm, anchor=south] (Case2) at (0,1) {using exactly three colors}; &
      \coloringVertexSpecial{red}{blue, thick}{green, line width=1.6pt}&
      \coloringEdgeSpecial{red}{green, line width=1.6pt}{blue, line width=1.6pt};
       &
      \\
  };
  \draw ([xshift=0.3em, yshift=2.7cm]Case1.north east) -- ([xshift=0.3em, yshift=-2cm]Case2.south east);
  \draw ([xshift=-3cm, yshift=-0.6em]m-1-1.south west) -- ([xshift=0.6em, yshift=-0.5em]m-1-3.south east);
\end{tikzpicture}
\end{center}
\caption{All special non-monochromatic colorings (up to isomorphism).}
\label{fig:SpecialColorings}
\end{figure}
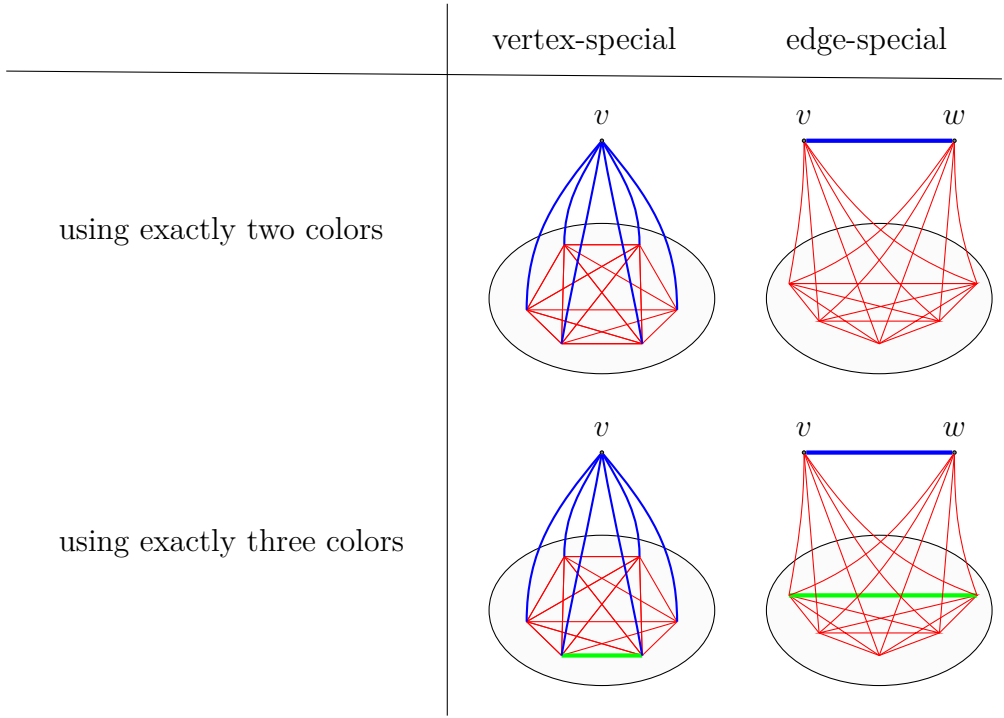

	Many of the lemmas presented here are dedicated to compute the number of extensions of some particular colorings. In order to keep the notation simple, whenever it is clear from the context, we keep the same name,~$\phi$, for the coloring of $E(K_n)$ and a particular extension $\phi : E(K_{n+1}) \to \{\red, \green, \blue\}$ of it.

\subsection{Graphs colored with exactly two colors}\label{sec:2col}

In the main result of this section (Lemma~\ref{lemma:twocolors}) we calculate a tight upper bound to the number of extensions of complete graphs that use exactly two colors.
  Before that, we compute the number of extensions of the special non-monochromatic $2$-colorings.

  \begin{lemma}\label{lemma:extremal_cases_two_colors}
    For all $n \geq 3$, if $\phi \in \Phi_n(2)$ is a special coloring of $K_n$ using exactly two colors, then
    $$
    w(\phi) = 3\cdot 2^{n-1} + 1.
    $$
  \end{lemma}

  \begin{proof}
    Let $n \geq 3$ be a fixed integer.
    Suppose  first that $\phi$ is the vertex-special coloring of $K_n$ with colors red and blue (see Figure~\ref{fig:SpecialColorings}) and $v \in K_n$ is its monochromatic vertex, say blue. Let $u$ be the new vertex added to $K_n$ to obtain $K_{n+1}$.
    We count the number of ways to color the edges from $u$ to $V(K_n)$, considering three cases according to the color of $uv$. (As mentioned earlier, we will also call $\phi$ the extension).

    \textbf{Case $\phi(uv) = \blue$.}
    Since we also have $\phi(vx) = \blue$ for every $x\in V(K_n)\setminus\{v\}$, there is no chance of having a rainbow triangle that uses $v$. And since $K_n - v$ is a monochromatic red graph, by Fact~\ref{claim:monochromatic_extensions}, we have $2^n - 1$ ways of coloring the edges from $u$ to $V(K_n)\setminus\{v\}$.

    \textbf{Case $\phi(uv) = \red$.}
    Since $\phi(vx) = \blue$ for every $x \in V(K_n)\setminus\{v\}$, then we cannot have any green edge from $u$ to $V(K_n)\setminus\{v\}$. Furthermore, there is no restriction about using colors red or blue from $u$ to $V(K_n)\setminus\{v\}$. Then, we have a total of $2^{n-1}$ ways of coloring the edges from $u$ to $V(K_n)\setminus\{v\}$.

    \textbf{Case $\phi(uv) = \green$.}
    For this case note that we cannot use color red on the edges between $u$ and $V(K_n)\setminus\{v\}$, so they are all blue or green. Recall that $K_n - v$ is a monochromatic red graph, and since $n \geq 3$, we have that $|V(K_n)\setminus\{v\}| \geq 2$.
    Therefore, for any two distinct vertices $y_1, y_2 \in V(K_n)\setminus\{v\}$, we must have $\phi(uy_1) = \phi(uy_2)$ as otherwise we would have a rainbow triangle. Thus all edges from $u$ to $V(K_n)\setminus\{v\}$ should have the same color (green or blue). 

    Then, we conclude that there are $w(\phi) = (2^{n} -1) + 2^{n-1} + 2 = 3\cdot 2^{n-1} + 1$ ways to extend the coloring $\phi$.

    Now, suppose  $\phi$ is the edge-special coloring of $K_n$ with the colors red and blue, and $b_1b_2 \in E(K_n)$ is its only blue edge. Let $u$ be the new vertex added to $K_n$ to obtain $K_{n+1}$.
    Similar to the vertex-special coloring, we consider cases according to the colors of the edges $ub_1$ and $ub_2$.

    \textbf{Case $\phi(ub_1) = \phi(ub_2) = \red$.}
    Note for this case that we do not have any restriction to the colors of the remaining edges incident to $u$, and as $K_n - b_1 - b_2$ is monochromatic, by Fact~\ref{claim:monochromatic_extensions}, there are $2^{n-1} - 1$ ways to color those remaining edges.

    \textbf{Cases ($\phi(ub_1) = \blue$ and $\phi(ub_2) = \green$) or ($\phi(ub_1) = \green$ and $\phi(ub_2) = \blue$).}
    Note for each of theses cases that there is only one way to color the edges from $u$ to $V(K_n)\setminus\{b_1, b_2\}$ because all those edges must be red. So, this gives us two other extensions.

    \textbf{Other cases.}
    There are other four cases for the colors of $ub_1$ and $ub_2$ (as the number of $3$-colored Gallai extensions of a single colored edge is seven). In each of those, we forbid the use of exactly one color to be used on the edges between $u$ and $V(K_n)\setminus\{b_1, b_2\}$. Furthermore, as opposed to the previously analyzed case, we can color freely such edges with the available two colors. Thus, for each of the four remaining cases, we have exactly $2^{n-2}$ ways to color the remaining edges.
\vspace{0.2cm}

    In total we have $w(\phi) = 2^{n-1} - 1 + 2 + 4\cdot 2^{n-2} = 3\cdot 2^{n-1} + 1$ ways to extend $\phi$.
  \end{proof}

   We will also use the following fact.

  \begin{fact}\label{claim:specialvertex}
    Let $k \geq 1$ and $n \geq \max\{4, 2k -1\}$ be fixed integers.
    For any coloring $\phi$ of $E(K_n)$ with exactly $k$ colors, there is a vertex $v$ such that $\phi^{\overline{v}}$ uses exactly $k$ colors.
  \end{fact}
  \begin{proof}
  	For $k=1$ the result is trivial. For $k=2$ and $n\ge 4$, clearly there must be a vertex~$x$ that is incident to edges of both colors; say $xa$ and $xb$ have different colors. Letting $v$ be any vertex in $V(K_n) - \{x, a, b\}$, we have that $\phi^{\overline{v}}$ uses both colors.

    So we may assume that $k \geq 3$ and $n \geq 2k -1$.
    Let $\phi$ be an arbitrarily coloring of $E(K_n)$ with $k$ colors.
    For a given $j \in \{1, \ldots, k\}$, we denote by $I_j$ the set of vertices $v$ of $V(K_n)$ such $\phi^{\overline{v}}$ does not use color $j$.
    Thus we want to show that $\bigcup_{j = 1}^k I_j \neq V(K_n)$. For each $u \in I_j$, all edges with color $j$ must be incident to $u$. Therefore, $|I_j| \leq 2$ and the only way to have $|I_j| = 2$ is when there is only one edge of color $j$.

    Suppose for a contradiction that $\bigcup_{i = 1}^k I_k = V(K_n)$.
    Since we have at least $2k - 1$ vertices and $|I_j| \leq 2$ for all $j \in \{1, \ldots, k\}$, at least $k-1$ of those sets must have size $2$. Assume without loss of generality that $|I_1|=\ldots=|I_{k-1}| = 2$. Then, the set $I_1\cup \ldots \cup I_{k-1}$ has at most $2k-2 \le n-1$ vertices and contains exactly one edge for each color in $\{1, \ldots, k-1\}$. Therefore, all other edges induced by $I_1\cup \ldots \cup I_{k-1}$ must have color $k$ (and as $k\ge 3$, this induced graph has more than $k-1$ edges). So we can take $v$ to be any vertex not in $I_1\cup \ldots \cup I_{k-1}$.
  \end{proof}

  Lemma~\ref{lemma:twocolors} below shows that the only extremal Gallai colorings that use exactly two colors are the special ones. In the proof of our main theorem (Theorem~\ref{thm:main_result}), for $\phi \in \Phi_n(2)$ we will only need to use that $w(\phi) \le 3\cdot 2^{n-1} + 1$. But we note that, instead of this, it is easier to prove that $w(\phi) < 3\cdot 2^{n-1}$ for non-special $2$-colorings, because of the way we apply induction.

	\begin{lemma}\label{lemma:twocolors}
		Given $n\ge 3$, let $\phi\colon E(K_n)\to\{\red,\blue\}$ be a non-monochromatic coloring.
		Then, the following hold:
		\begin{enumerate}
		\item  If $\phi$ is special, then $w(\phi) = 3\cdot 2^{n-1} + 1$.
		\item  If $\phi$ is non-special, then $w(\phi) < 3\cdot 2^{n-1}$.
		\end{enumerate}
	\end{lemma}
	\begin{proof}
    If $\phi$ is special, then we are done by Lemma~\ref{lemma:extremal_cases_two_colors}. For the non-special colorings, we proceed by induction on $n$.
  	
  	In what follows consider a non-special (in particular, non-monochromatic) $2$-coloring $\phi \colon E(K_n)\to\{\red,\blue\}$. Since every Gallai $2$-coloring of $K_3$ is special, we may assume that $n\geq 4$. Our inductive step will only work for $n\ge 5$, so we also need to treat $n = 4$ as a base case.
  	
  	Let $n = 4$ and $\phi$ be a non-special coloring of $K_4$ that uses colors red and blue. We will show that $\phi$ has to be isomorphic to one of the colorings depicted in Figure~\ref{fig:2coloringsK4}. Assuming this, we have that $w(\phi) \le 23 < 3\cdot2^3$ (see Lemma~\ref{lemma:case_n_4} in the appendix).

\begin{figure}[thb]
\begin{center}
  \begin{tikzpicture}[scale=1.7]
  \pgfsetlinewidth{1pt}

  \tikzset{vertex/.style={circle, minimum size=0.7cm, fill=black!20, draw, inner sep=1pt}}

  \begin{scope}[xshift=0, yshift=0]
    \node [vertex] (c) at (0,1) {$x_1$};
    \node [vertex] (b) at (1,1) {$x_2$};
    \node [vertex] (d) at (0,0) {$x_3$};
    \node [vertex] (a) at (1,0) {$x_4$};

    \draw[edgeblue] (b) -- (c) -- (d) -- (b);
    \draw[edgeblue] (a) -- (c);
    \draw[edgered] (a) -- (b);
    \draw[edgered] (a) -- (d);

    \draw[black] (0.5,-0.2) node [below] {\ref{case_n_4-a}};
   \end{scope}

   \begin{scope}[xshift=2cm, yshift=0]
    \node [vertex] (c) at (0,1) {$x_1$};
    \node [vertex] (b) at (1,1) {$x_2$};
    \node [vertex] (d) at (0,0) {$x_3$};
    \node [vertex] (a) at (1,0) {$x_4$};

    \draw[edgeblue] (d) -- (c) -- (a) -- (b);
    \draw[edgered] (c) -- (b) -- (d) -- (a);

	\draw[black] (0.5,-0.2) node [below] {\ref{case_n_4-b}};
   \end{scope}

   \begin{scope}[xshift=4cm, yshift=0]
    \node [vertex] (b) at (1,1) {$x_2$};
    \node [vertex] (d) at (0,0) {$x_3$};
    \node [vertex] (c) at (0,1) {$x_1$};
    \node [vertex] (a) at (1,0) {$x_4$};

    \draw[edgeblue] (a) -- (b) -- (d) -- (c) -- (a);
    \draw[edgered] (b) -- (c)(a) -- (d);

    \draw[black] (0.5,-0.2) node [below] {\ref{case_n_4-c}};
   \end{scope}
  \end{tikzpicture}
\caption{All non-special $2$-colorings of $K_4$ (up to isomorphism). The number extension of each of them is computed in Lemma~\ref{lemma:case_n_4}.}
\label{fig:2coloringsK4}
\end{center}
\end{figure}
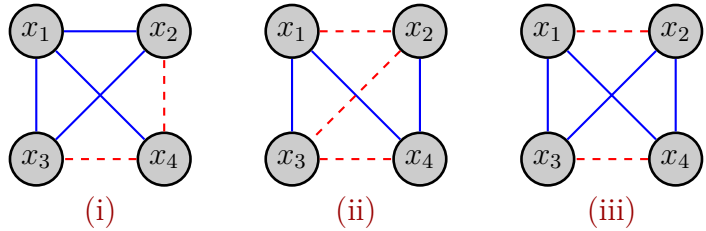

 Let $V(K_4) = \{x_1, x_2, x_3, x_4\}$. If there is a monochromatic vertex, say blue, then the other three vertices form a triangle with two red edges and one blue edge, as otherwise $\phi$ would be special. So we have the coloring depicted in Figure~\ref{fig:2coloringsK4}-\ref{case_n_4-a}. Thus, we may assume that there is no monochromatic vertex. Since $x_1$ is not monochromatic, we may assume without loss of generality that $x_1x_2$ is red, and $x_1x_3$ and $x_1x_4$ are blue. Now, $x_2$ must have at least one blue neighbor and, by symmetry, we may assume that $x_2x_4$ is blue. Now $x_4$ must have a red neighbor and the only option is $x_3$. Finally, depending on the color of $x_2x_3$ we either have the coloring depicted in Figure~\ref{fig:2coloringsK4}-\ref{case_n_4-b} or the one in Figure~\ref{fig:2coloringsK4}-\ref{case_n_4-c}.
	
     For the inductive step, suppose that $n\geq 5$ and that the result holds for any non-monochromatic $2$-coloring of the edges of $K_{n-1}$.
     Recall that $\phi$ is a non-monochromatic coloring of $E(K_n)$ with colors red and blue.
     By Fact~\ref{claim:specialvertex}, there exists a vertex $v$ such that $\phi^{\overline{v}}$ is not monochromatic.
  	 If $w(\phi^{\overline{v}}) < 3\cdot 2^{n-2} $ then, by Lemma~\ref{lemma:2t+1}, we have $w(\phi) \le 2(3\cdot 2^{n-2} - 1) + 1 = 3\cdot 2^{n-1} - 1$.
  	 Thus, we may assume that $w(\phi^{\overline{v}}) \ge 3\cdot 2^{n-2}$.
  	 Then, from the inductive hypothesis, we conclude that $\phi^{\overline{v}}$ is special.
  	 We will consider separately the cases where $\phi^{\overline{v}}$ is vertex-special and edge-special.

     Suppose first $\phi^{\overline{v}}$ is vertex-special and, assume without loss of generality, that there is a blue vertex $w$ in $\phi^{\overline{v}}$.
     Let $\{x_1,\dots,x_{n-2}\} = V(K_n) \setminus \{v, w\}$.
     As $n\ge 5$, the coloring $\phi^{\overline{x_i}}$ is not monochromatic for any $1\leq i\leq n-2$ (even in $K_n - \{v, x_i\}$ there are two edges of different colors). As before, by Lemma~\ref{lemma:2t+1}, we conclude that $\phi^{\overline{x_i}}$ is special, for any $1\leq i\leq n-2$. This implies that $\phi(vw)=\blue$, and  $\phi(vx_i) = \phi(vx_j)$ for any $1\leq i<j\leq n-2$.
     If $\phi(vx_i)=\red$ for every $i$, then $\phi$ would be vertex-special ($w$ would be a blue vertex and every other edge would be red), a contradiction.
     Thus, we may assume that $\phi(vx_i)=\blue$ for every $i$.
     Note that, since $\phi^{\overline{x_i}}$ is extremal, this is possible only if $n=5$, and the coloring is the one depicted in Figure~\ref{fig:2coloringK5-a} (with $\{v,w\} = \{y_1, y_2\}$).
     But for such coloring, we conclude that $w(\phi) = 45 < 3\cdot 2^4$ (see Lemma~\ref{lemma:case_n_5} in the appendix).

\begin{figure}[thb]
\begin{center}
  \begin{tikzpicture}[scale=1.3]
  \pgfsetlinewidth{1pt}

  \tikzset{vertex/.style={circle, minimum size=0.7cm, fill=black!20, draw, inner sep=1pt}}

  \begin{scope}[xshift=0, yshift=0]
  \foreach \i/\label in {1/x_1, 2/x_2, 3/x_3, 4/y_1, 5/y_2}{
    \node [vertex] (x\i) at (\i*72.00+18:1cm) {$\label$};
  }

	\draw[edgered] (x1) -- (x2) -- (x3) -- (x1);
	\draw[edgeblue] (x4) -- (x5);
	  \foreach \i in {1, 2, 3} {
		  \foreach \j in {4, 5}{
		    \draw[edgeblue] (x\i) -- (x\j);
		  }
	  }
   \draw[black] (0,-1.2) node [below] {};
   \end{scope}

  \end{tikzpicture}
\caption{Coloring of $K_5$ isomorphic to the one of Lemma~\ref{lemma:case_n_5}.}
\label{fig:2coloringK5-a}
\end{center}
\end{figure}
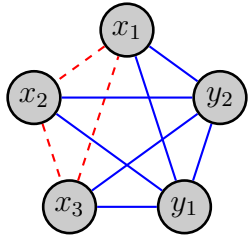

  	At last, suppose $\phi^{\overline{v}}$ is edge-special and assume without loss of generality that there are vertices $y$ and $z$ with $\phi^{\overline{v}}(yz)=\blue$ and all other edges of $K_n-v$ are colored red.
  	Let $x_1,\dots,x_{n-3}$ be the other vertices of $K_n$.
      Similarly as before, we can easily conclude that $\phi^{\overline{x_i}}$ is not monochromatic and, therefore, is special for any $1\leq i\leq n-3$.
      Note that since $\phi^{\overline{x_i}}$ is special, we have $\phi(vx_i)=\red$, and $\phi(vy)=\phi(vz)$.
      If $\phi(vy)=\phi(vz) = \red$, then $\phi$ would be edge-special ($yz$ would be the only blue edge of $\phi$), a contradiction.
      Thus, we may assume that $\phi(vy)=\phi(vz)=\blue$. Note that, since $\phi^{\overline{x_i}}$ is extremal, this is possible only if $n=5$. Again, for such coloring, from Lemma~\ref{lemma:case_n_5}, we conclude that $w(\phi) = 45 < 3\cdot 2^4$.
	\end{proof}

\subsection{Graphs colored with exactly three colors}\label{sec:3col}

In the main result of this section (Lemma~\ref{lemma:threecolors}) we calculate a tight upper bound to the number of extensions of complete graphs that use exactly three colors.
As in Section~\ref{sec:2col}, we start by computing the number of extensions of special colorings on three colors.

\begin{lemma}\label{lemma:extremal_cases_three_colors}
  For all $n \geq 4$, if $\phi \in \Phi_n(3)$ is a special coloring of $K_n$ using exactly three colors, then
  $$
  w(\phi) =  2^n + 3.
  $$
\end{lemma}

\begin{proof}
  Let $n\geq 4$ be a fixed integer.
  Suppose first that $\phi \in \Phi_n(3)$ is a vertex-special coloring, $v \in V(K_n)$ is its monochromatic vertex, say blue, and $g_1g_2 \in E(K_n)$ is a green edge such that all the remaining edges are red.
  Let $u$ be the new vertex added to $K_n$ to obtain $K_{n+1}$.
  Similarly to the proof of Lemma~\ref{lemma:extremal_cases_two_colors}, we count the number of ways to extend $\phi$ by considering cases according to the color of $uv$.

  \textbf{Case $\phi(uv) = \blue$.}
  Since the vertex $v$ is a monochromatic blue vertex, the fact that $\phi(uv) = \blue$ does not restrict the choices of the colors for the remaining edges.
  On the other hand, $\phi^{\overline{v}}$ is the edge-special coloring for two colors.
  Thus, by Lemma~\ref{lemma:extremal_cases_two_colors}, there are $3\cdot 2^{n -2} + 1$ ways to color the remaining edges.

  \textbf{Case $\phi(uv) = \red$.}
  Note that we will be able to use only colors red and blue on the edges from $u$ to $V(K_n)\setminus\{v\}$. This already implies that no rainbow triangle contains $v$. Note that $ug_1$ and $ug_2$ must have the same color, but there are no extra restrictions to color the edges between $u$ and $V(K_n) - v - g_1 - g_2$. Therefore, there are $2 \cdot 2^{n-3}$ colorings for this case.

  \textbf{Case $\phi(uv) = \green$.}
  All edges between $u$ and $V(K_n)\setminus\{v\}$ must be blue or green. Note that, as all edges between $\{g_1, g_2\}$ and $V(K_n)\setminus\{g_1, g_2, v\}$ are red, $\phi(ug_1) = \phi(ug_2)$, as otherwise we would have no color choice for the edge $uy$, for any $y \in  V(K_n)\setminus\{g_1, g_2, v\}$ (recall that $n\ge 4)$.
  On the other hand, once we choose the color of $\phi(ug_1)$ the color of $uy$ has to be the same for every $y \in  V(K_n)\setminus\{g_1, g_2, v\}$. Thus, there are only two ways of coloring the remaining edges.
Therefore, we conclude that $w(\phi) = (3\cdot2^{n-2} + 1) + 2^{n-2} + 2 = 2^n + 3$.

  Suppose that $\phi \in \Phi_n(3)$ is the edge-special coloring.
  Let $b_1b_2 \in E(K_n)$ be the blue edge and $g_1g_2 \in E(K_n)$ be the green edge such that all the remaining edges are red. Let $u$ be the new vertex added to $K_n$ to obtain $K_{n+1}$. We consider the (seven) possible ways to color the edges $ub_1$ and $ub_2$ (so that the triangle $ub_1b_2$ is not rainbow). On the description of each case we list the values of $(\phi(ub_1), \phi(ub_2))$ considered.

  \textbf{Case $(\red, \red)$:} No matter how we color the edges between $u$ and $K_n - b_1 - b_2$, there will be no rainbow triangle that contains the vertices $b_1$ or $b_2$, so there is no extra restriction on the choices of the colors of those edges. Since $K_n - b_1 - b_2$ is the edge-special coloring on two colors, there are $3\cdot 2^{n-3}+1$ ways to color the edges between $u$ and $K_n - b_1 - b_2$.

  \textbf{Cases $(\blue, \blue)$ or $(\blue, \red)$ or $(\red, \blue)$:} Each of these configurations only forbids us to use the color green on the remaining edges.
  On the other hand, for all possible ways to color the remaining edges using colors red and blue, the restriction is that $ug_1$ and $ug_2$ must have the same color. Therefore, there are $2\cdot 2^{n-2} = 2^{n-3}$ ways to color them.

  \textbf{Case $(\green, \green)$:}
  This configuration forbids us to choose the color blue for any remaining edges. But the coloring $\phi$ restricted to the graph $K_n - b_1 - b_2$ uses only colors red and green, so using red and green for edges between $u$ and $V(K_n) - b_1 - b_2$ does not create any rainbow triangle. Thus, there are $2^{n-2}$ ways to color the remaining edges.

  \textbf{Cases $(\blue,\green)$ or $(\green, \blue)$:}
  Each of these configurations forbids us to choose both green and blue for the remaining edges. Thus, there is only one way to color the remaining edges.
  \vspace{0.2cm}

In total, we obtain $w(\phi) = (3 \cdot 2^{n-3} + 1) + 3\cdot 2^{n-3} + 2^{n-2} + 2 = 2^n + 3$.
\end{proof}

  Lemma~\ref{lemma:threecolors} below is the analogue of Lemma~\ref{lemma:twocolors} for three colors, and its proof is also based on Lemma~\ref{lemma:2t+1} and induction on $n$. It shows that the special colorings are the only extremal ones in $\Phi_n(3)$. As with  Lemma~\ref{lemma:twocolors}, in the proof of Theorem~\ref{thm:main_result} we will only need that $w(\phi) \le 2^n + 3$ for every $\phi \in \Phi_n(3)$. However, using induction, it is easier to prove that $w(\phi) < 2^n$ for non-special colorings in $\Phi_n(3)$.

  We will use the following simple fact.

  \begin{fact}\label{fact:isoK4}
  Every Gallai coloring of $K_4$ that uses exactly three colors is isomorphic to a special coloring. The special colorings of $K_4$ are depicted in Figure~\ref{fig:K4-specials}.
  \end{fact}

  \begin{figure}
  \begin{center}
  \begin{tikzpicture}[scale=1.4]
  \tikzset{vertex/.style={circle, minimum size=0.6cm, fill=black!20, draw, inner sep=1pt}}
  \begin{scope}
  \foreach \i/\l in {1/a, 2/b, 3/c,4/d}{
      \node[vertex] (x\i) at (72*\i+90:1) {$\l$};
  }
  \draw[edgeblue] (x1) --(x2);
  \draw[edgeblue] (x1) --(x3);
  \draw[edgeblue] (x1) --(x4);
  \draw[edgegreen] (x4)--(x3);
  \draw[edgered] (x4) -- (x2) -- (x3);
  \draw[black] (0,-1.2) node [below] {(a) Vertex-special };
  \end{scope}
  \begin{scope}[xshift=4.5cm]
  \foreach \i/\l in {1/a, 2/b, 3/c,4/d}{
      \node[vertex] (x\i) at (72*\i+90:1) {$\l$};
  }
  \draw[edgeblue] (x1) edge (x4);
  \draw[edgegreen] (x2)--(x3);
  \draw[edgered] (x1)--(x2) -- (x4) -- (x3) -- (x1);
  \draw[black] (0,-1.2) node [below] {(b) Edge-special};
  \end{scope}
  \end{tikzpicture}
  \end{center}
  \caption{All special colorings of $K_4$ that use exactly three colors (up to isomorphism). These are particular cases of the second row of Figure~\ref{fig:SpecialColorings}.}
  \label{fig:K4-specials}
  \end{figure}
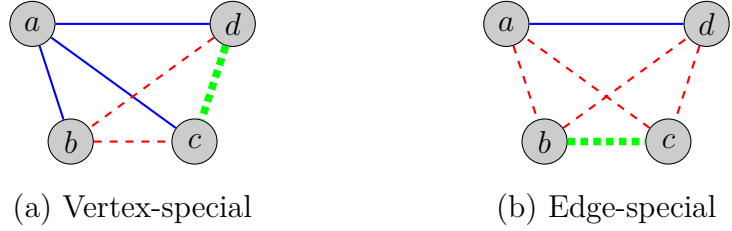

  \begin{lemma}\label{lemma:threecolors}
  Given $n\ge 4$, let $\phi\colon E(K_n)\to\{\red,\blue,\green\}$ be a Gallai coloring with exactly three colors.
  Then, the following hold:
    \begin{enumerate}
      \item If $\phi$ is special, then $w(\phi) = 2^n + 3$.
      \item If $\phi$ is non-special, then $w(\phi) < 2^n$.
    \end{enumerate}
  \end{lemma}

  \begin{proof}
  If $\phi$ is special, then we are done by Lemma~\ref{lemma:extremal_cases_three_colors}. For the rest we use induction on $n$. So, assume that $\phi \in \Phi_n(3)$ is a non-special Gallai coloring.
	
  By Fact~\ref{fact:isoK4}, we may assume $n\ge 5$. For the inductive step to work we will need to have $n\ge 6$.
  The proof of the base case $n=5$ is tedious (but not so long), and we postpone it to the appendix.
  Lemma~\ref{lemma:ExtensionsOf-K5-nonspecial} shows that, for every non-special coloring $\phi \in \Phi_5(3)$, we have $w(\phi) \le 31 < 2^5$ .

  Now, suppose that $n\geq 6$ and that Lemma~\ref{lemma:threecolors} holds for any Gallai colorings in $\Phi_{n-1}(3)$.
   By Fact~\ref{claim:specialvertex}, there exists a vertex $v$ such that $\phi^{\overline{v}}$ still uses exactly three colors.
	 If $w(\phi^{\overline{v}}) < 2^{n-1} $ then, by Lemma~\ref{lemma:2t+1}, we have $w(\phi) \le 2( 2^{n-1} - 1) + 1 =  2^{n} - 1$.
	 Thus, we may assume that $w(\phi^{\overline{v}}) \ge 2^{n-1}$.
	 Then, by the inductive hypothesis, as $\phi^{\overline{v}}$ uses all three colors, we conclude that $\phi^{\overline{v}}$ is special.
	 We will consider separately the cases where $\phi^{\overline{v}}$ is vertex-special and edge-special.

  Suppose first $\phi^{\overline{v}}$ is vertex-special and assume without loss of generality that there is a blue vertex $w$ in $\phi^{\overline{v}}$, a unique green edge $g_1g_2$, and let $x_1,\dots,x_{n-4}$ be the other vertices of $K_n$ (note that since $n\geq 6$, the number of vertices in $x_1,\dots,x_{n-4}$ is at least two).
  Since $\phi^{\overline{x_i}}$ uses all three colors for any $1\leq i\leq n-4$, we know that $\phi^{\overline{x_i}}$ is special, as otherwise we would have $w(\phi^{\overline{x_i}}) < 2^{n-1}$ and Lemma~\ref{lemma:2t+1} would finish the proof.
  But for $\phi^{\overline{x_i}}$ to be edge-special, there should be at least two colors that appear each only once in $\phi^{\overline{x_i}}$, and this is not the case for colors red and blue (as $n\ge 6$).
	Thus we may assume that $\phi^{\overline{x_i}}$ is vertex-special for any $1\leq i\leq n-4$. But note that then $w$ should be the monochromatic (blue) vertex of $\phi^{\overline{x_i}}$ for any $1\leq i\leq n-4$, from where we conclude that $\phi(vx_i)=\red$ for any $1\leq i\leq n-4$.
	Therefore, $\phi$ is vertex-special (with blue vertex $w$), a contradiction with the fact that $\phi$ is non-special.

	At last, suppose $\phi^{\overline{v}}$ is edge-special and assume without loss of generality that there are vertices $b_1$ and $b_2$ joined by the only blue edge, vertices $g_1$ and $g_2$ joined by the only green edge, and let $x_1,\dots,x_{n-5}$ be the other vertices of $K_n$ (note that since $n\geq 6$ there is at least one vertex in $x_1,\dots,x_{n-5}$).
	Similarly as before (using Lemma~\ref{lemma:2t+1}) we can conclude that $\phi^{\overline{x_i}}$ is special for any $1\leq i\leq n-5$.
	For $\phi^{\overline{x_i}}$ to be vertex-special, there should be a monochromatic vertex $w$ in $\phi^{\overline{x_i}}$ with color $c\in\{\red,\blue,\green\}$ such that $c$ does not appear in $\phi^{\overline{w}}$, but this is not possible.
	Thus we may assume that $\phi^{\overline{x_i}}$ is edge-special for $1\leq i\leq n-5$.
	But note that then $v$ should be monochromatic in red in the coloring $\phi^{\overline{x_i}}$.
	If $n\geq 7$, then there are at least two vertices in $x_1,\dots,x_{n-5}$ and we know that $\phi(vx_1)=\ldots=\phi(vx_{n-5})=\red$, from where we conclude that $\phi$ is edge-special, a contradiction.
	On the other hand, if $n=6$, then it is not clear what is the color of the edge $vx_1$.
	Clearly, if $\phi(vx_1)=\red$, then we get that $\phi$ is edge-special, obtaining the desired contradiction.
	Thus, we may assume that $\phi(vx_1)\in\{\green,\blue\}$, in which case we obtain the coloring depicted in Figure~\ref{fig:2coloringK6}. Therefore, $w(\phi) = 53 < 2^6 + 3$ as necessary (by Lemma~\ref{lemma:case_6} in the appendix).

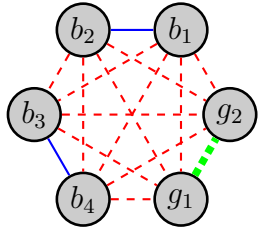
\begin{figure}[thb]
\begin{center}
  \begin{tikzpicture}[scale=1.3]
  \pgfsetlinewidth{1pt}

  \tikzset{vertex/.style={circle, minimum size=0.7cm, fill=black!20, draw, inner sep=1pt}}

    \begin{scope}[xshift=3cm, yshift=0]
	\foreach \i/\label in {1/b_1, 2/b_2, 3/b_3, 4/b_4, 5/g_1, 6/g_2}{
	\node [vertex] (x\i) at (\i*60.00:1cm) {$\label$};
	}

	  \foreach \i in {1, ..., 5} {
	  \pgfmathsetmacro\result{int(\i+1)}
		  \foreach \j in {\result, ..., 6}{
		    \draw[edgered] (x\i) -- (x\j);
		  }
	  }
  	\draw[color=white, line width=2mm] (x1) -- (x2);
  	\draw[edgeblue] (x1) -- (x2);
  	\draw[color=white, line width=2mm] (x3) -- (x4);
	\draw[edgeblue] (x3) -- (x4);
	\draw[color=white, line width=2mm] (x5) -- (x6);
	\draw[edgegreen] (x5) -- (x6);

   \draw[black] (0,-1.2) node [below] {};
   \end{scope}

  \end{tikzpicture}
\caption{Coloring of $K_6$ isomorphic to the one of Lemma~\ref{lemma:case_6}.}
\label{fig:2coloringK6}
\end{center}
\end{figure}

\end{proof}

\subsection{Proof of Theorem~\ref{thm:main_result}}\label{sec:proof-maintheorem}

Let $c_i(n) = |\Phi_n(i)|$, for $i\in \{1, 2, 3\}$,  be the number of Gallai colorings of $K_n$ that use exactly $i$ colors. Then,
\begin{equation}\label{eq:cn123}
	c(n) = c_1(n) + c_2(n) + c_3(n).
\end{equation}

Clearly, $c_1(n) = 3$ for every $n\ge 2$, and $c_2(n) = 3 \cdot (2^{\binom{n}{2}} - 2)$ for every $n\ge 3$. Also notice that $c_3(1) = c_3(2) = c_3(3) = 0$. In particular, $c(2) = 3$ and $c(3) = 21$.

So, our aim is to estimate $c_3(n)$ for $n\ge 4$. Fix $i\in \{1,2,3\}$.
For each coloring $\phi$ in $\Phi_n(i)$ count the number of extensions of $\phi$ such that the resulting coloring of $E(K_{n+1})$ uses all three colors, and let $w_i(n)$ be the maximum of those numbers.

Since there is no way to extend a monochromatic coloring of $K_n$ to a Gallai coloring of $K_{n+1}$ that uses all three colors, $w_1(n) = 0$ for every $n$.
Using Lemma~\ref{lemma:twocolors}, for $n\ge 3$, we obtain $w_2(n) \leq (3\cdot 2^{n-1} + 1) - 2^n = 2^{n-1} + 1$. In fact, for any non-monochromatic $2$-coloring of $K_n$, say with red and blue, its number of extensions is at most $3\cdot 2^{n-1} + 1$ and among those there are at least $2^n$ extensions that are valid extensions but use only colors red and blue. Using Lemma~\ref{lemma:threecolors}, for $n\ge 4$, we have $w_3(n) \leq 2^n + 3$. And because $c_3(3) =0$, the following holds for every $n\ge 3$:
\begin{align}
	c_3(n+1) & \leq w_1(n)c_1(n) + w_2(n)c_2(n) + w_3(n)c_3(n) \nonumber                                  \\
	         & \leq 0 + (2^{n-1} +1)\cdot (3 \cdot 2^{\binom{n}{2}}-6) + (2^n + 3)\cdot c_3(n) \label{eq:c3nrecurrenceUgly} \\
           & \leq 3\cdot (2^{n-1}+1)\cdot 2^{\binom{n}{2}}+(2^n+3)\cdot c_3(n)\label{eq:c3nrecurrence}.
\end{align}

We claim that $c_3(n) \le 7n\, 2^{\binom{n}{2}}$ for every $n$. Setting $n=3$ at inequality~\eqref{eq:c3nrecurrenceUgly} gives $c_3(4) \le 18\cdot 5 = 90 < 7\cdot 4\cdot 2^{\binom{4}{2}}$. So, the claim holds for every $n\le 4$. Now assume that it holds for some particular $n\ge 4$ and let us show that it holds for $n+1$.

From inequality~\eqref{eq:c3nrecurrence}, we have
\begin{align}\label{eq:x}
c_3(n+1) & \le 3\cdot (2^{n-1}+1)\cdot 2^{\binom{n}{2}}+(2^n+3)\cdot 7n\cdot 2^{\binom{n}{2}}\nonumber \\
        & \le 7(n+1)\, 2^{\binom{n+1}{2}}.
\end{align}
To see that~\eqref{eq:x} holds, note that $2^{\binom{n+1}{2}} = 2^n\, 2^{\binom{n}{2}}$.
Finally, using \eqref{eq:cn123}, we have
\begin{align*}
  c(n) & \le 3+\left(3\cdot 2^{\binom{n}{2}} -6\right) + 7n\, 2^{\binom{n}{2}} \\
  & \le 3\cdot 2^{\binom{n}{2}} + 7n\, 2^{\binom{n}{2}} \\
       & \le 7(n+1)\, 2^{\binom{n}{2}}.
\end{align*}

\begin{remark*} To keep the proof simple, we did not optimize the constant multiplying $(n+1)$ above. However, this is not hard to do. If we consider a function $f(n)$ such that $f(1)=f(2)=f(3)=0$ and that satisfies the analogous of \eqref{eq:c3nrecurrence} with equality, that is,
\[
  f(n+1) = 3\cdot 2^n\cdot 2^{\binom{n}{2}}+(2^n+3)f(n),
\]
it is trivial to see (by induction) that $c_3(n)\le f(n)$ for every $n$. To solve such recursion we can use the substitution $f(n) = k(n)\, 2^{\binom{n}{2}}$. We get the linear recursion
\[
  k(n+1) = 3+ \left(1+\frac{3}{2^n}\right) k(n),
\]
that can be solved exactly with standard methods, noting also that the product $\prod_{n=3}^{\infty} \left(1+\frac{3}{2^n}\right)$ is convergent. It turns out that $k(n)$ is asymptotically approximately $t n$, for some constant~$t < 7$. Similarly, we could also define $f(n)$ using inequality \eqref{eq:c3nrecurrenceUgly} instead of \eqref{eq:c3nrecurrence}, but again, this would one would only improve this bound slightly.
\end{remark*}

\section{Concluding remarks}

 Recall that $c_i(n) = |\Phi_n(i)|$, for $i\in \{1, 2, 3\}$, is the number of Gallai colorings of $K_n$ that use exactly $i$ colors from a total of $3$ colors.
  In Table~\ref{table:cn} below we show the exact values of $c_1(n)$, $c_2(n)$, $c_3(n)$, and $c(n)$ for $n\le 8$.
These values were obtained by enumerating (with a computer program) all Gallai colorings of $K_n$.
 
The previous best known upper bound on the number of Gallai colorings of $K_n$ was asymptotically $(n-1)!\,2^{\binom{n}{2}}$.
In Theorem~\ref{thm:main_result} we improve this substantially showing that the number of Gallai $3$-colorings of $K_n$ is at most $7(n+1)2^{\binom{n}{2}}$.
But since the lower bound is $3\cdot 2^{n\choose 2}-3$, there is still a large gap between the lower and upper bounds.
Even though Table~\ref{table:cn-2} considers only very small values of $n$, it may indicate that $c(n)$ is asymptotically closer to the lower bound. However, in Table~\ref{table:cn-2}, we compute the ratios of our upper bound to $c(n)$ and of $c(n)$ to that lower bound. We found, surprisingly, that neither of those are monotone. Again, this could simply be some artifact due to the fact that $n$ is very small, but it could also indicate that there is a term multiplying $2^{\binom{n}{2}}$ that is not a constant. It would be interesting to fully understand the behavior of $c(n)$ (even only for large $n$). It would also be interesting to study what happens when we consider more colors.

 \begin{table}[ht]
  \begin{tabular}{c|crrr}
  \toprule
  $n$ & $c_1(n)$ & $c_2(n)$    & $c_3(n)$    & $c(n)$\\
  \midrule
  2 & 3 & -           & -           & 3               \\
  3 & 3 & 18          & -           & 21                 \\
  4 & 3 & 186         & 90          & 279               \\
  5 & 3 & 3,066       & 3,060       & 6,129              \\
  6 & 3 & 98,298      & 112,686     & 210,987      \\
  7 & 3 & 6,291,450   & 5,522,496   & 11,813,949    \\
  8 & 3 & 805,306,362 & 407,207,826 & 1,212,514,191 \\ 
\bottomrule
\end{tabular}
\vspace{0.2cm}
  \caption{Values of $c_1(n)$, $c_2(n)$, $c_3(n)$, and $c(n)$.}
  \label{table:cn}
\end{table}

  \begin{table}[ht]
  \begin{tabular}{c|rrr|r|r}
  \toprule
  $n$ & $c(n)$  & $3\cdot 2^{n\choose 2}-3$ & $7(n+1)2^{\binom{n}{2}}$ & $7(n+1)2^{\binom{n}{2}}/ c(n)$ & $c(n)/\big(3\cdot 2^{n\choose 2} - 3\big)$\\
  \midrule
  2 & 3         & 3   & 42                  & 14.00 & 1.00\\
  3 & 21        & 21    & 224                 & 10.66 & 1.00\\
  4 & 279         & 189 & 2,240               &  8.02 & 1.47\\
  5 & 6,129       & 3,069 & 43008             &  7.01 & 1.99\\
  6 & 210,987       & 98,301  & 1,605,632         &  7.61 & 2.14\\
  7 & 11,813,949    & 6,291,453 & 117,440,512     &  9.94 & 1.87\\
  8 & 1,212,514,191     & 805,306,365 & 16,911,433,728  & 13.94 & 1.50\\
\bottomrule
\end{tabular}
\vspace{0.2cm}
  \caption{Comparison between $c(n)$ and the lower and upper bounds.}
  \label{table:cn-2}
\end{table}

\bibliographystyle{amsplain}

\section{Appendix}\label{sec:appendix}

Here, we compute the number of Gallai extensions of some particular colorings of $K_4$ and $K_5$. We also describe all Gallai $3$-colorings of $K_5$ that do not have a monochromatic vertex (up to isomorphism). In most proofs of this appendix we use repeatedly, and sometimes implicitly, the following trivial fact.

  \begin{fact}\label{fact:trivial}
  Let $c\in\{\red,\blue\}$ and $\phi\colon E(K_n)\to\{\red,\blue\}$. Consider an extension of $\phi$ to a Gallai coloring of $E(K_{n+1})$ with colors red, green, or blue. Let $u$ be the vertex added to $K_n$ to obtain $K_{n+1}$.
  If $vw$ is an edge of the initial $K_n$, say with color blue, and $\phi(uv)=\green$, then $\phi(uw)$ cannot be red.
  In particular, if there are vertices $v$ and $w$ in the initial $K_n$ such that all edges between $v$ and $K_n-v$ are blue and $\phi(uw) = \green$, then there are no red edges between $u$ and $K_n-w$.
  \end{fact}

 Let us start with the colorings of $E(K_4)$ isomorphic to those depicted in Figure~\ref{fig:2coloringsK4}.
 For convenience, we show Figure~\ref{fig:2coloringsK4} again below.

\begin{figure}[thb]
\begin{center}
  \begin{tikzpicture}[scale=1.9]
  \pgfsetlinewidth{1pt}

  \tikzset{vertex/.style={circle, minimum size=0.7cm, fill=black!20, draw, inner sep=1pt}}

  \begin{scope}[xshift=0, yshift=0]
    \node [vertex] (c) at (0,1) {$x_1$};
    \node [vertex] (b) at (1,1) {$x_2$};
    \node [vertex] (d) at (0,0) {$x_3$};
    \node [vertex] (a) at (1,0) {$x_4$};

    \draw[edgeblue] (b) -- (c) -- (d) -- (b);
    \draw[edgeblue] (a) -- (c);
    \draw[edgered] (a) -- (b);
    \draw[edgered] (a) -- (d);

    \draw[black] (0.5,-0.2) node [below] {\ref{case_n_4-a}};
   \end{scope}

   \begin{scope}[xshift=2cm, yshift=0]
    \node [vertex] (c) at (0,1) {$x_1$};
    \node [vertex] (b) at (1,1) {$x_2$};
    \node [vertex] (d) at (0,0) {$x_3$};
    \node [vertex] (a) at (1,0) {$x_4$};

    \draw[edgeblue] (d) -- (c) -- (a) -- (b);
    \draw[edgered] (c) -- (b) -- (d) -- (a);

	\draw[black] (0.5,-0.2) node [below] {\ref{case_n_4-b}};
   \end{scope}

   \begin{scope}[xshift=4cm, yshift=0]
    \node [vertex] (b) at (1,1) {$x_2$};
    \node [vertex] (d) at (0,0) {$x_3$};
    \node [vertex] (c) at (0,1) {$x_1$};
    \node [vertex] (a) at (1,0) {$x_4$};

    \draw[edgeblue] (a) -- (b) -- (d) -- (c) -- (a);
    \draw[edgered] (b) -- (c)(a) -- (d);

    \draw[black] (0.5,-0.2) node [below] {\ref{case_n_4-c}};
   \end{scope}
  \end{tikzpicture}
\end{center}
\end{figure}

\begin{lemma}\label{lemma:case_n_4}
Let $c\in\{\red,\blue\}$ and consider a coloring $\phi\colon E(K_4)\to\{\red,\blue\}$.
Then the following hold:
\begin{enumerate}[label=(\roman*)]
	\item \label{case_n_4-a} If there are only two edges with color $c$ and they form a path, then $w(\phi)=23$ (see Figure~\ref{fig:2coloringsK4}-\ref{case_n_4-a}).
	\item \label{case_n_4-b} If there are only three edges with color $c$ and they form a path, then $w(\phi)=21$ (see Figure~\ref{fig:2coloringsK4}-\ref{case_n_4-b}).
	\item \label{case_n_4-c} If there are only two edges with color $c$ and they form a matching, then $w(\phi)=23$ (see Figure~\ref{fig:2coloringsK4}-\ref{case_n_4-c}).
\end{enumerate}
\end{lemma}

\begin{proof}
	The proofs of all items are similar and simple, but we show them here for completeness. Let $\phi\colon E(K_4)\to\{\red,\blue\}$ and without loss of generality assume $c=\red$. Let $v$ be a new vertex and add all edges between $v$ and the four vertices of $K_4$. Let us count in how many ways we can color the edges incident to $v$, with colors red, blue, and green without creating rainbow triangles.

	Clearly, there are exactly $2^4$ possible ways to color the edges incident to $v$ without using color green. By Fact~\ref{fact:trivial}, there are exactly four ways in which we use exactly one green edge (we can only choose which edge receives color green, and the color of the remaining edges are forced). 
	Furthermore, clearly there is only one way in which all four edges are green. This adds up to $21$ possible extensions. It remains to count the number of extensions such that there are two or three green edges incident to $v$.\ \\

	\noindent\textit{Proof of item~\ref{case_n_4-a}}.
	Let $\{x_1,x_2,x_3\}$ be a blue triangle where $x_1$ is a blue vertex, and let $x_4$ be the other vertex of $K_4$, so $\phi(x_2x_4)=\phi(x_4x_3)=\red$. By Fact~\ref{fact:trivial}, there is only one way to have exactly three green edges incident to $v$, which is with $\phi(vx_2)=\phi(vx_3)=\phi(x_4v)=\green$ (and then $\phi(vx_1)$ is forced to be blue). Moreover, the only way to have exactly two green edges incident to $v$ is with $\phi(vx_2)=\phi(vx_3)=\green$ (the colors of the other edges are forcibly determined and valid). This gives, in total, $2+21 = 23$ extensions of $\phi$ to a Gallai coloring.\ \\

	\noindent\textit{Proof of item~\ref{case_n_4-b}}.
	Let $W \subset V(K_4)$ be any set with two or three vertices. Note that there is always a vertex $y \in V(K_4) \setminus W$ such that $y$ has a blue and a red neighbor in $W$. So, if the green edges incident to $v$ are exactly those with an endpoint in $W$, by Fact~\ref{fact:trivial}, there is no color available for $vy$. Therefore, there is no way of extending $\phi$ to a Gallai coloring using exactly two or three green edges incident to $v$. So there are only the $21$ previous extensions to a Gallai coloring.\ \\

	\noindent\textit{Proof of item~\ref{case_n_4-c}}. Let $\{x_1,x_2,x_3,x_4\}$ be the set of vertices of $K_4$ such that $x_1x_2$ and $x_3x_4$ are the red edges, and all other edges are blue.
	Let $W \subset V(K_4)$ be any set with exactly three vertices. As in the previous item, the vertex $y\in V(K_4)\setminus W$ has a red and a blue neighbor in $W$, so there is no extension of $\phi$ that uses exactly three green edges incident to $v$.
	On the other hand, if we want exactly two green edges incident to $v$, then there are exactly two possible ways: $\phi(vx_1)=\phi(vx_2)=\green$ or $\phi(vx_3)=\phi(vx_4)=\green$ (the colors of the other edges are forcibly determined). This gives, in total, $2+21 = 23$ extensions of $\phi$ to a Gallai coloring.
\end{proof}

Lemma~\ref{lemma:case_n_5} provides the number of extensions of colorings isomorphic the one in Figures~\ref{fig:2coloringK5-a}.

\begin{lemma}\label{lemma:case_n_5}
Let $c\in\{\red,\blue\}$ and consider a coloring $\phi\colon E(K_5)\to\{\red,\blue\}$.
If there are only three edges with color $c$ and they form a triangle, then $w(\phi)=45$.
\end{lemma}

\begin{proof}
Let $\phi\colon E(K_5)\to\{\red,\blue\}$ be a coloring as in the statement of this lemma and assume without loss of generality that $c=\red$. Let $X=\{x_1,x_2,x_3\}$ be the vertices that form the red triangle, $Y=\{y_1,y_2\}$ the other two vertices of $K_5$, and let $v$ be a new vertex adjacent to all other vertices. Let us show that there are $45$ ways to color the edges incident to $v$, with colors red, blue, and green without creating rainbow triangles.

First of all notice that there are exactly $32$ possible ways to color the edges incident to $v$ without using green edges. Now assume that there is at least one green edge incident to $v$.
If there is a green edge between $v$ and $Y$, then there are no red edges between $v$ and $X$, by Fact~\ref{fact:trivial}.
But since the edges inside $X$ are red, we conclude that $\phi(vx_1)=\phi(vx_2)=\phi(vx_3)$. 
Thus, there are only $2$ possibilities of colors for the edges between $v$ and $X$ (they are all blue or all green). Since there are $3$ ways to color the edges between $v$ and $Y$ using color green at least once, this yields $6$ extensions of $\phi$. One may check that all $6$ extensions are valid.

It remains to count the number of extensions with no green edges between $v$ and $Y$ and at least one green edge between $v$ and $X$. In this case the edges between $v$ and $Y$ must be blue, and those between $v$ and $X$ must be red or green (with at least one green edge). Then, there are $7$ possible extensions of $\phi$.
Therefore, in total there are $32+6+7 = 45$ extensions of $\phi$ to a Gallai coloring.
\end{proof}

The structure of the proof of the next result, Lemma~\ref{lemma:case_K5_3colors_MonocVertex}, is very similar to the one of Lemma~\ref{lemma:extremal_cases_three_colors}, but we include its proof here for completeness.

\begin{lemma}\label{lemma:case_K5_3colors_MonocVertex}
Let $\phi$ be a non-special Gallai coloring of $K_5$ that uses exactly three colors and has a monochromatic vertex $v$. Then, $w(\phi) \le 31$.
\end{lemma}
\begin{proof}
  Let $\phi$ be a non-special Gallai coloring of $K_5$ that uses exactly three colors and let $v \in V(K_n)$ be a monochromatic vertex, say blue.
  Suppose first that $\phi^{\overline{v}}$ uses all three colors. Recall that the only two colorings of $K_4 = K_5-v$ that use three colors are the special colorings. So, by Lemma~\ref{lemma:extremal_cases_three_colors}, we have $w(\phi^{\overline{v}}) = 2^4+3 = 19$.

  Now, we follow the same steps as in the proof for vertex-special colorings in Lemma~\ref{lemma:extremal_cases_three_colors}. We add a new vertex $u$ and we want to count in how many way we can color all edges between $u$ and $V(K_n)$.
 As before, we consider the same cases according to the color of $uv$.

  \textbf{Case $\phi(uv) = \blue$.} This imposes no restriction on the color of the edges between $v$ and $V(K_5) \setminus \{v\}$. Therefore, the number of extensions in this case is simply $w(\phi^{\overline{v}}) = 2^4+3 = 19$.

  \textbf{Case $\phi(uv) = \red$.}
  We are able to use only colors red and blue on the edges between $u$ and $V(K_n)\setminus\{v\}$.
  This would give at most $2^{4}$ extensions. However, for each green edge, say $g_1g_2$, induced by $V(K_n)\setminus\{v\}$, the colors of $ug_1$ and $ug_2$ must be the same (otherwise $u,g_1,g_2$ would be rainbow). As we have at least one green edge, $w(\phi^{\overline{v}})\leq 2^3$, with equality if and only if there is only one green edge.

  \textbf{Case $\phi(uv) = \green$.}
  Similar to the previous case, we obtain $w(\phi^{\overline{v}})\leq 2^3$, with equality only if there is only one red edge induced by $V(K_n)\setminus\{v\}$.
\vspace{0.2cm}

  This would give a total of $19+8+8 = 35$ extensions, but to have exactly this amount, it must be the case where there is only one red and only one green edge. In this case, we would have an edge-special coloring, that is not possible. Therefore, either there are two green or two red edges, in which case the number of extensions of $\phi$ is at most $19+8+4 = 31$.

  Now, assume that $\phi^{\overline{v}}$ uses only two colors. As all edges incident to $v$ are blue, and we must use all three colors, we must use colors red and green on $\phi^{\overline{v}}$. It is a well-known fact that a $2$-colored $K_n$ must have a monochromatic spanning tree. Suppose without loss of generality  that such spanning tree is red. Furthermore, notice that if there is only one green edge, then the original coloring, $\phi$, on $K_5$ will be vertex-special and that is not possible.

   \textbf{Case $\phi(uv) = \blue$.} The number of extensions in this case is simply $w(\phi^{\overline{v}})$, which by Lemma~\ref{lemma:extremal_cases_two_colors} is at most $3\cdot 2^{3}+1 = 25$.

   \textbf{Case $\phi(uv) = \red$.}
   We are able to use only colors red and blue on the edges between $u$ and $V(K_n)\setminus\{v\}$. This would give at most $2^{4}$ extensions. But, since $\phi^{\overline{v}}$ has at least two green edges, there are at most $2^{2} = 4$ ways to color the edges between $u$ and $V(K_n)\setminus\{v\}$.

  \textbf{Case $\phi(uv) = \green$.} All edges between $u$ and $V(K_n)\setminus\{v\}$ must be blue or green. The fact that $\phi^{\overline{v}}$ has a red spanning tree implies that all edges between $u$ and $V(K_n)\setminus\{v\}$ must receive the same color. So, there are only two ways to color those edges.
\vspace{0.2cm}

  In total, there are at most $25+4+2 = 31$ extensions.
\end{proof}

The colorings of $K_5$ depicted in Figure~\ref{fig:3coloringK5} are important for the proof of the next result, Lemma~\ref{lemma:List_of_K5-nonspecial}.

\newcounter{isocounter}\renewcommand{\theisocounter}{\Alph{isocounter}}
\begin{figure}[thb]
\begin{center}
  \begin{tikzpicture}[scale=1.4]
  \pgfsetlinewidth{1pt}

  \tikzset{vertex/.style={circle, minimum size=0.6cm, fill=black!20, draw, inner sep=1pt}}

  \begin{scope}[yshift=0cm]
    \begin{scope}[xshift=0cm]
      \foreach \i/\label in {1/x_1, 2/x_2, 3/x_3, 4/x_4, 5/x_5}{
        \node [vertex] (x\i) at (\i*72+18:1cm) {$\label$};
      }
      \foreach \i/\j in {1/5, 1/4, 2/3, 2/4, 2/5} {
      \draw[edgeblue] (x\i) -- (x\j);
      }
      \foreach \i/\j in {1/2, 1/3, 3/4, 3/5} {
      \draw[edgered] (x\i) -- (x\j);
      }
      \foreach \i/\j in {4/5} {
      \draw[edgegreen] (x\i) -- (x\j);
      }
      \refstepcounter{isocounter}\label{isotypesA}
      \draw[black] (0,-1) node [below] {(\theisocounter)};
    \end{scope}

    \begin{scope}[xshift=3cm]
      \foreach \i/\label in {1/x_1, 2/x_2, 3/x_3, 4/x_4, 5/x_5}{
        \node [vertex] (x\i) at (\i*72+18:1cm) {$\label$};
      }
      \foreach \i/\j in {1/3, 1/4, 1/5, 2/3, 2/4, 2/5} {
      \draw[edgeblue] (x\i) -- (x\j);
      }
      \foreach \i/\j in {1/2, 3/4, 3/5} {
      \draw[edgered] (x\i) -- (x\j);
      }
      \foreach \i/\j in {4/5} {
      \draw[edgegreen] (x\i) -- (x\j);
      }
      \refstepcounter{isocounter}\label{isotypesB}
      \draw[black] (0,-1) node [below] {(\theisocounter)};
    \end{scope}

    \begin{scope}[xshift=6cm]
      \foreach \i/\label in {1/x_1, 2/x_2, 3/x_3, 4/x_4, 5/x_5}{
        \node [vertex] (x\i) at (\i*72+18:1cm) {$\label$};
      }
      \foreach \i/\j in {1/3, 1/4, 1/5, 2/3, 2/4, 2/5} {
      \draw[edgeblue] (x\i) -- (x\j);
      }
      \foreach \i/\j in {3/4, 3/5} {
      \draw[edgered] (x\i) -- (x\j);
      }
      \foreach \i/\j in {1/2, 4/5} {
      \draw[edgegreen] (x\i) -- (x\j);
      }
      \refstepcounter{isocounter}\label{isotypesC}
      \draw[black] (0,-1) node [below] {(\theisocounter)};
    \end{scope}
  \end{scope}
  \begin{scope}[yshift=-3.2cm]

    \begin{scope}[xshift=1.5cm]
      \foreach \i/\label in {1/x_1, 2/x_2, 3/x_3, 4/x_4, 5/x_5}{
        \node [vertex] (x\i) at (\i*72+18:1cm) {$\label$};
      }
      \foreach \i/\j in {3/4, 3/5} {
      \draw[edgeblue] (x\i) -- (x\j);
      }
      \foreach \i/\j in {1/3, 1/4, 1/5, 2/3, 2/4, 2/5, 4/5} {
      \draw[edgered] (x\i) -- (x\j);
      }
      \foreach \i/\j in {1/2} {
      \draw[edgegreen] (x\i) -- (x\j);
      }
      \refstepcounter{isocounter}\label{isotypesD}
      \draw[black] (0,-1) node [below] {(\theisocounter)};
    \end{scope}

    \begin{scope}[xshift=4.5cm]
      \foreach \i/\label in {1/x_1, 2/x_2, 3/x_3, 4/x_4, 5/x_5}{
        \node [vertex] (x\i) at (\i*72+18:1cm) {$\label$};
      }
      \foreach \i/\j in {3/4, 4/5, 5/3} {
      \draw[edgeblue] (x\i) -- (x\j);
      }
      \foreach \i/\j in {1/3, 1/5, 1/4, 2/3, 2/4, 2/5} {
      \draw[edgered] (x\i) -- (x\j);
      }
      \foreach \i/\j in {1/2} {
      \draw[edgegreen] (x\i) -- (x\j);
      }
      \refstepcounter{isocounter} \label{isotypesE}
      \draw[black] (0,-1) node [below] {(\theisocounter)};
    \end{scope}
  \end{scope}

  \end{tikzpicture}
\caption{Colorings of $K_5$ that are non-special and have no monochromatic vertex.}
\label{fig:3coloringK5}
\end{center}
\end{figure}
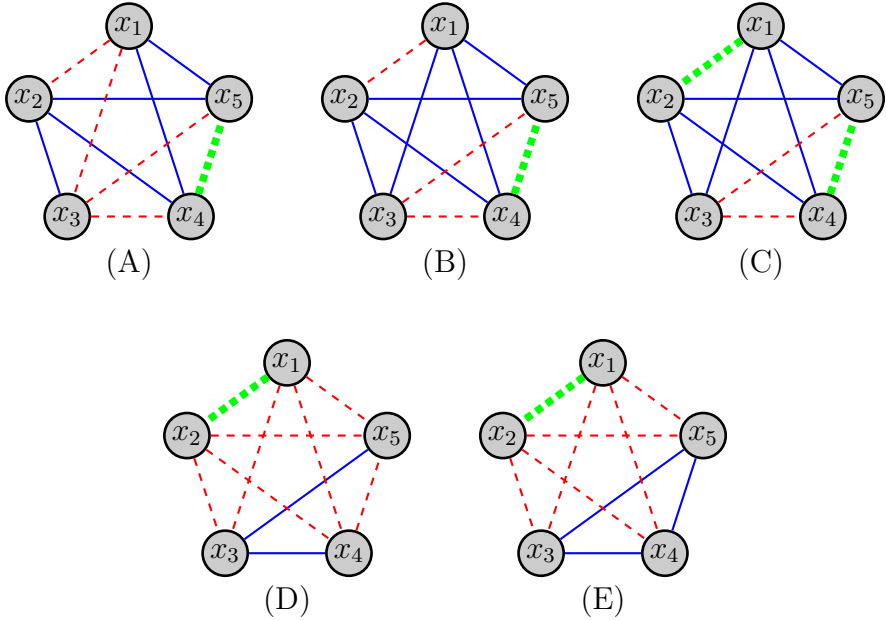

\begin{lemma}\label{lemma:List_of_K5-nonspecial}
Every non-special Gallai coloring of $K_5$ that uses exactly three colors and does not have a monochromatic vertex is isomorphic to one of the colorings depicted in Figure~\ref{fig:3coloringK5}.
\end{lemma}
\begin{proof}
By Fact~\ref{claim:specialvertex}, there exists a vertex $v$ such that $\phi^{\overline{v}}$ uses all three colors. Since $K_5-v$ has only $4$ vertices, it follows that $\phi^{\overline{v}}$ is a special coloring.
Let $V(K_5 - v) = \{a, b, c, d\}$. By Lemma~\ref{lemma:extremal_cases_three_colors}, we have $w(\phi^{\overline{v}}) = 19$. This means that there are $19$ ways to color the edges between $v$ and $\{a, b, c, d\}$ for every Gallai coloring of $V(K_5 - v)$ with exactly $3$ colors.

 We consider separately the cases where $\phi^{\overline{v}}$ is vertex-special or edge-special. Assume first that $\phi^{\overline{v}}$ is vertex-special and, without loss of generality, that $a$ is monochromatic in blue (inside $K_4$), $bd$ is green, and the other two edges are red (see Figure~\ref{fig:K4-specials}).
 We partition the set of the $19$ extensions of $\phi^{\overline{v}}$ into three classes (as in the proof of Lemma~\ref{lemma:extremal_cases_three_colors}), according to the color of the edge $va$. Following the counting in the proof of Lemma~\ref{lemma:extremal_cases_three_colors}, for $va$ being blue, red, or green, we have, respectively, $13$, $4$, and $2$ extensions ($19 = 13 + 4 + 2$). If $va$ were blue, then $v$ would be a monochromatic vertex in $\phi$, a contradiction (so the first $13$ extensions are not feasible). If $va$ is red, then $vb$, $vc$, and $vd$ must be either red of blue. This implies that $\phi(vb) = \phi(vd)$ (as $bd$ is green). Among the four options for the colors of $vb$ and $vd$, there is one in which $v$ becomes monochromatic in red. From the other three colorings, two are isomorphic to the coloring in Figure~\ref{fig:3coloringK5}-(\ref{isotypesA}) (swapping colors red with blue in one of the colorings) and the third one is the coloring \ref{fig:3coloringK5}-(\ref{isotypesB}). Finally, assume that $va$ is green. In this case, $vb$, $vc$, and $vd$ must be either green or blue and must all have the same color. From these two options, one of them makes $v$ become monochromatic (green). The other one is given in Figure~\ref{fig:3coloringK5}-(\ref{isotypesC}).

\newcounter{allCounter}
\newcommand{\KfourES}[1]{
  \foreach \i in {0, ...,4}{
      \node[vertex] (x\i) at (72*\i+90:1) {};
  }
  \draw[edgeblue] (x1) edge (x4);
  \draw[edgegreen] (x2)--(x3);
  \draw[edgered] (x1)--(x2) -- (x4) -- (x3) -- (x1);
  \refstepcounter{allCounter};
  \node[label={[label distance=-0cm]above:$v$}] at (x0) {};
  \draw[#1] (0,-1.2) node [below] {(\theallCounter)};
}
\newcommand{\xUmQuatro}[2]{%
    \draw[edge#1](x0)--(x1);\draw[edge#2](x0)--(x4);
}
\newcommand{\xDoisTres}[2]{%
    \draw[edge#1](x0)--(x2);\draw[edge#2](x0)--(x3);
}

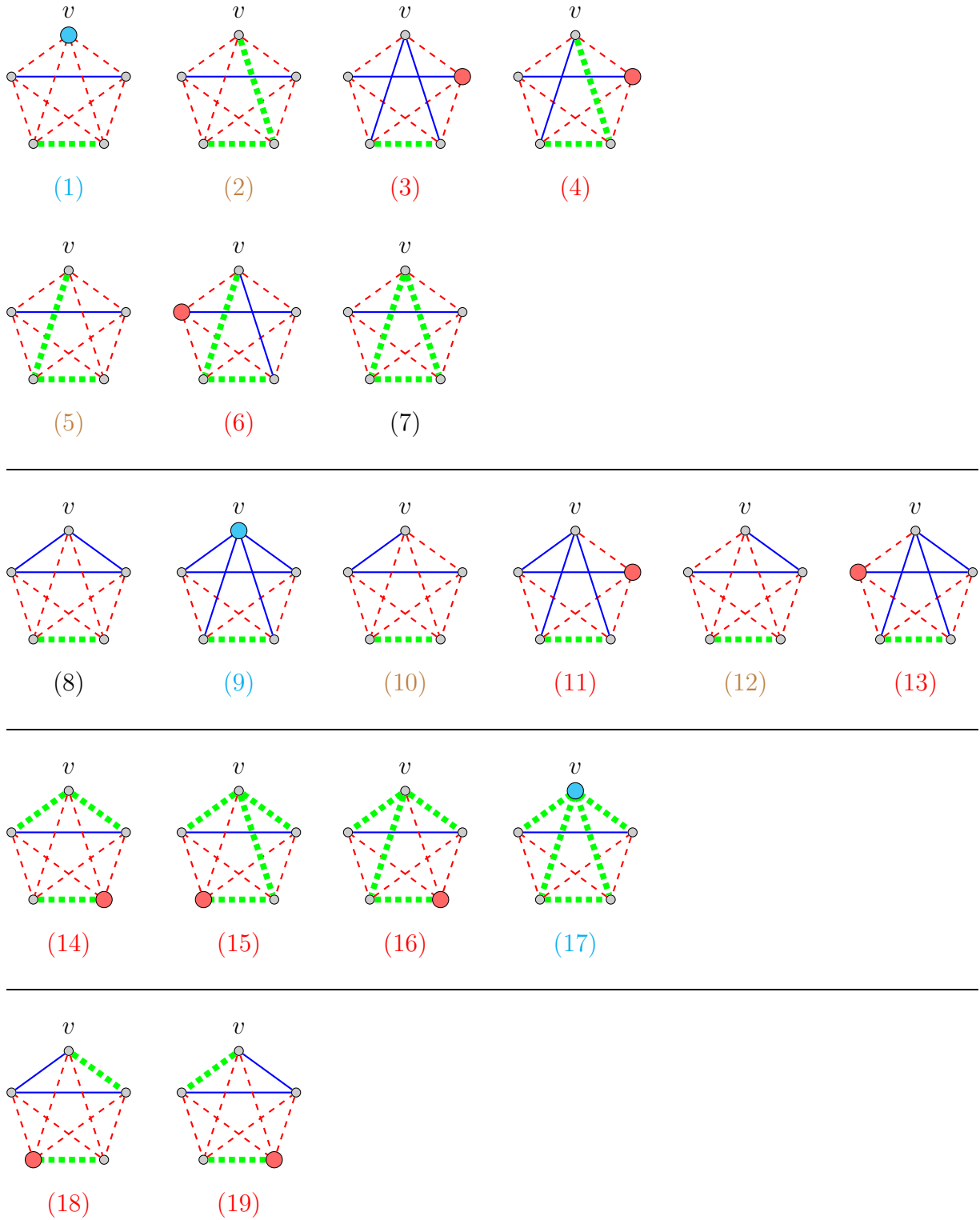
\begin{figure}
\begin{center}
\begin{tikzpicture}
\tikzset{vertex/.style={circle, minimum size=0.15cm, fill=black!20, draw, inner sep=1pt}}
\tikzset{K4monoVrtx/.style={circle, minimum size=0.27cm, fill=red!60, draw, inner sep=2pt}
}
\tikzset{k5monoVrtx/.style={circle, minimum size=0.27cm, fill=cyan!60, draw, inner sep=2pt}}
\tikzset{hasK5monoVrtx/.style={color=cyan}}
\tikzset{hasK4monoVrtx/.style={color=red}}
\tikzset{isotypeC/.style={color=brown}}
\tikzset{isotypeF/.style={color=black}}

  \matrix[column sep=0.7cm,row sep=0.4cm] {
     \KfourES{hasK5monoVrtx}; \label{fig:K5mono1}
     \xUmQuatro{red}{red} \xDoisTres{red}{red}      \node[k5monoVrtx] at (x0){};&

     \KfourES{isotypeC}          \label{isotypeC:1}
     \xUmQuatro{red}{red} \xDoisTres{red}{green}&

     \KfourES{hasK4monoVrtx}  \label{fig:K4mono1}
     \xUmQuatro{red}{red} \xDoisTres{blue}{blue}    \node[K4monoVrtx] at (x4){};&

     \KfourES{hasK4monoVrtx}  \label{fig:K4mono2}
     \xUmQuatro{red}{red} \xDoisTres{blue}{green}   \node[K4monoVrtx] at (x4){};\\

     \KfourES{isotypeC}          \label{isotypeC:2}
     \xUmQuatro{red}{red} \xDoisTres{green}{red}&

     \KfourES{hasK4monoVrtx}  \label{fig:K4mono3}
     \xUmQuatro{red}{red} \xDoisTres{green}{blue}   \node[K4monoVrtx] at (x1){};&

     \KfourES{isotypeF}          \label{isotypeF:1}
     \xUmQuatro{red}{red} \xDoisTres{green}{green}&
     \\ \hline \\

     \KfourES{isotypeF}          \label{isotypeF:2}
     \xUmQuatro{blue}{blue} \xDoisTres{red}{red}&

     \KfourES{hasK5monoVrtx}  \label{fig:K5mono2}
     \xUmQuatro{blue}{blue} \xDoisTres{blue}{blue}   \node[k5monoVrtx] at (x0){};&

     \KfourES{isotypeC}          \label{isotypeC:3}
     \xUmQuatro{blue}{red} \xDoisTres{red}{red}&

     \KfourES{hasK4monoVrtx} \label{fig:K4mono4}
     \xUmQuatro{blue}{red} \xDoisTres{blue}{blue}   \node[K4monoVrtx] at (x4){};&

     \KfourES{isotypeC} \label{isotypeC:4}
     \xUmQuatro{red}{blue}; \xDoisTres{red}{red}&

     \KfourES{hasK4monoVrtx} \label{fig:K4mono5}
     \xUmQuatro{red}{blue}; \xDoisTres{blue}{blue}   \node[K4monoVrtx] at (x1){};
     \\ \hline \\

     \KfourES{hasK4monoVrtx} \label{fig:K4mono6}
     \xUmQuatro{green}{green} \xDoisTres{red}{red}   \node[K4monoVrtx] at (x3){};&

     \KfourES{hasK4monoVrtx} \label{fig:K4mono7}
     \xUmQuatro{green}{green} \xDoisTres{red}{green}  \node[K4monoVrtx] at (x2){};&

     \KfourES{hasK4monoVrtx} \label{fig:K4mono8}
     \xUmQuatro{green}{green} \xDoisTres{green}{red}   \node[K4monoVrtx] at (x3){};&

     \KfourES{hasK5monoVrtx} \label{fig:K5mono3}
     \xUmQuatro{green}{green} \xDoisTres{green}{green}   \node[k5monoVrtx] at (x0){};
     \\ \hline \\

     \KfourES{hasK4monoVrtx} \label{fig:K4mono9}
     \xUmQuatro{blue}{green} \xDoisTres{red}{red}  \node[K4monoVrtx] at (x2){};&

     \KfourES{hasK4monoVrtx} \label{fig:K4mono10}
     \xUmQuatro{green}{blue} \xDoisTres{red}{red}  \node[K4monoVrtx] at (x3){};\\
  };
\end{tikzpicture}
\end{center}
\caption{Colorings of $K_5$ with a vertex $v$ such that $K_5-v$ is a $3$-colored edge-special $K_4$.}\label{fig:K5_extending_edgeSpecialK4}
\end{figure}

Now, consider the case where $\phi^{\overline{v}}$ is edge-special, say $ab$ is blue, $cd$ is green, and the other four edges are red. Again, following the cases in Lemma~\ref{lemma:extremal_cases_three_colors}, we may partition the set of extensions of $\phi^{\overline{v}}$ according to the colors of the edges $va$ and $vb$. Lemma~\ref{lemma:extremal_cases_three_colors} gives $7 + 3\cdot 2 + 4 + 2 = 19$ extensions. All of them are shown in Figure~\ref{fig:K5_extending_edgeSpecialK4}. Three of them, (\ref{fig:K5_extending_edgeSpecialK4}-(\ref{fig:K5mono1}), \ref{fig:K5_extending_edgeSpecialK4}-(\ref{fig:K5mono2}), and \ref{fig:K5_extending_edgeSpecialK4}-(\ref{fig:K5mono3})), have a monochromatic vertex (so $\phi$ cannot be one of them). Ten of them, (\ref{fig:K5_extending_edgeSpecialK4}-(\ref{fig:K4mono1}), \ref{fig:K5_extending_edgeSpecialK4}-(\ref{fig:K4mono2}), \ref{fig:K5_extending_edgeSpecialK4}-(\ref{fig:K4mono3}), \ref{fig:K5_extending_edgeSpecialK4}-(\ref{fig:K4mono4}), \ref{fig:K5_extending_edgeSpecialK4}-(\ref{fig:K4mono5}), \ref{fig:K5_extending_edgeSpecialK4}-(\ref{fig:K4mono6}), \ref{fig:K5_extending_edgeSpecialK4}-(\ref{fig:K4mono7}), \ref{fig:K5_extending_edgeSpecialK4}-(\ref{fig:K4mono8}), \ref{fig:K5_extending_edgeSpecialK4}-(\ref{fig:K4mono9}), and \ref{fig:K5_extending_edgeSpecialK4}-(\ref{fig:K4mono10})), have a vertex $z$ (that we marked with a different color in Figure~\ref{fig:K5_extending_edgeSpecialK4}) such that $\phi^{\overline{z}}$ is isomorphic to a vertex-special $K_4$ on $3$ colors. Therefore, they were already treated in the case where $\phi^{ \overline{v}}$ is edge-special. Four of the remaining colorings, (\ref{fig:K5_extending_edgeSpecialK4}-(\ref{isotypeC:1}), \ref{fig:K5_extending_edgeSpecialK4}-(\ref{isotypeC:2}), \ref{fig:K5_extending_edgeSpecialK4}-(\ref{isotypeC:3}), and \ref{fig:K5_extending_edgeSpecialK4}-(\ref{isotypeC:4})), are isomorphic to Figure~\ref{fig:3coloringK5}-(\ref{isotypesD}). And the remaining two colorings, (\ref{fig:K5_extending_edgeSpecialK4}-(\ref{isotypeF:1}) and \ref{fig:K5_extending_edgeSpecialK4}-(\ref{isotypeF:2})), are isomorphic to the one in Figure~\ref{fig:3coloringK5}-(\ref{isotypesE}).
\end{proof}

\begin{lemma}\label{lemma:ExtensionsOf-K5-nonspecial}
Every non-special Gallai coloring of $K_5$ that uses exactly three colors has at most $31$ extensions.
\end{lemma}
\begin{proof}
Let $\phi \in \Phi_5(3)$ be a non-special Gallai coloring of $K_5$ that uses exactly three colors. If $K_5$ has a monochromatic vertex, we are done by Lemma~\ref{lemma:case_K5_3colors_MonocVertex}. So, assume this is not the case. By Lemma~\ref{lemma:List_of_K5-nonspecial}, we only need to compute the number of extensions for each one of the colorings in Figure~\ref{fig:3coloringK5}. We let $V(K_5) = \{x_1,\ldots, x_5\}$ as in Figure~\ref{fig:3coloringK5}.

We split the colorings into five types, \ref{isotypesA}, \ref{isotypesB}, \ref{isotypesD}, \ref{isotypesC}, and \ref{isotypesE} (following the labels in Figure~\ref{fig:3coloringK5}).
Let $u$ be a vertex not in $K_5$. In the colorings of type \ref{isotypesB}, \ref{isotypesD}, \ref{isotypesC}, and \ref{isotypesE} there exists an edge, $x_1x_2$, such that all edges from $\{x_1, x_2\}$ to $\{x_3,x_4,x_5\}$ have the same color. So we can first color the edges $ux_1$ and $ux_2$, and then use Fact~\ref{fact:trivial} to restrict the colors allowed for the edges from $u$ to $\{x_3,x_4, x_5\}$. We organize the seven ways to color the edges $ux_1$ and $ux_2$ into four cases, following the same structure as in the proof of the edge-special case of Lemma~\ref{lemma:extremal_cases_three_colors}. To keep it compact, we will write sentences of the form ``case $\{rr, bg\}$'' to mean $(\phi(ux_1), \phi(ux_2)) \in \{(\red, \red), (\blue, \green)\}$.

\setlist[description]{leftmargin=0.8cm,labelindent=\parindent}
\begin{description}
  \item[Extensions of (\ref{isotypesB})] For the cases $\{bb\}$, $\{rb, rr, br\}$, $\{gg\}$, and $\{gr, rg\}$, in this order, we add to a total of $(3\cdot 2^2+1) + 3\cdot 4 + 2 + 2\cdot 1 =  29$ extensions. For case $\{bb\}$ we used Lemma~\ref{lemma:extremal_cases_two_colors}, and for case $\{gg\}$ we used that there is a red spanning tree in $\{x_3,x_4, x_5\}$ (so all edges from $u$ to this set are either green or blue, hence they must have the same color).

  \item[Extensions of (\ref{isotypesC})] For the cases $\{bb\}$, $\{gb, gg, bg\}$, $\{rr\}$, and $\{rg, gr\}$, in this order, we add to a total of $(3\cdot 2^2+1) + 3\cdot 2 + 4 + 2\cdot 1 = 25  $ extensions. For case $\{bb\}$ we used Lemma~\ref{lemma:extremal_cases_two_colors}, and for cases $\{gb, gg, bg\}$ we used that $\{x_3,x_4, x_5\}$ has a red spanning tree.

  \item[Extensions of (\ref{isotypesD})] For the cases $\{rr\}$, $\{gr, gg, rg\}$, $\{bb\}$, and $\{bg, gb\}$, in this order, we add to a total of $(3\cdot 2^2+1) + 3\cdot 2 + 8 + 2\cdot 1 = 29$ extensions. For case $\{rr\}$, we have used Lemma~\ref{lemma:extremal_cases_two_colors}. For the cases $\{gr, gg, rg\}$ we have used that $\{x_3,x_4, x_5\}$ has a spanning blue tree.

  \item[Extensions of (\ref{isotypesE})] For the cases $\{rr\}$, $\{gr, gg, rg\}$, $\{bb\}$, and $\{bg, gb\}$, in this order, we add to a total of $(2^4-1) + 3\cdot 2 + 8 + 2\cdot 1 = 31$ extensions. For case $\{rr\}$, we have used Fact~\ref{claim:monochromatic_extensions}, and for case $\{gr, gg, rg\}$ we used that $\{x_3,x_4, x_5\}$ has a blue spanning tree.
\end{description}

Finally, a coloring of type \ref{isotypesA} has to be treated in an ad-hoc way.

\begin{description}
  \item[Extensions of (\ref{isotypesA})] We consider cases depending on the colors of $ux_4$ and $ux_5$. For $i\in \{1, 2, 3\}$, we denote $\phi(ux_i)$ by $c_i$ and consider $c_i \in \{\red, \green, \blue\}$.
  \begin{description}
    \item[Case $(\phi(ux_4), \phi(ux_5)) = (\red,\red)$] This imposes no restriction on the colors $c_1$ and $c_2$, but forces $c_3\in \{\blue, \red\}$. If none of $c_1$, $c_2$, and $c_3$ is green, we have a valid Gallai coloring (as $\{x_1, \ldots,x_4\}$ also only has red and blue edges). This gives $2^3 = 8$ colorings. It remains to check the cases where either $c_1$ or $c_2$ is green. If $c_1$ is green, then we must have $c_2 = c_3 = \red$. If $c_2$ is green, then $c_1=\red$ and $c_3 = \blue$. Furthermore, it is not possible to have $c_1=c_2=\green$ (as there would be no color available for $c_3$). This gives a total of $8+2 = 10$ colorings.

    \item[Case $(\phi(ux_4), \phi(ux_5)) \in \{\green,\green)$] Analogously to the previous case, there is no restriction for $c_3$, but we must have $\{c_1, c_2\} \subseteq \{\red, \blue\}$. If none of $c_1$, $c_2$ and $c_3$ is green, we have $2^3 = 8$ valid colorings. Otherwise, $c_3$ is green and this forces $c_1=\red$ and $c_2=\blue$. So we have $9$ colorings in this case.

    \item[Cases in $(\phi(ux_4), \phi(ux_5)) \in \{(\green,\blue), (\blue,\green)\}$] The green edge (among $ux_4$ and $ux_5$) forces that each of $c_1$ and $c_2$ must be green or blue. While the fact that $\{\phi(ux_4), \phi(ux_5)\} = \{\green, \blue\}$ forces $c_3 = \red$. Now $(c_3=\red) \implies (c_2=\blue)$ and  $(c_2=\blue) \implies (c_1 = \blue)$. So, for each of the two choices for $ux_4$ and $ux_5$ there is only one way to complete the coloring. This gives us $2$ colorings.

    \item[Cases in $(\phi(ux_4), \phi(ux_5)) \in \{(\red,\green), (\green,\red)\}$] The colors of $ux_4$ and $ux_5$ already imply that $c_1=c_2=\blue$ and $c_3 \in \{\red, \green\}$. But now, $(c_1\blue) \implies (c_3 = \red)$. As in the previous case, we have $2$ colorings.
    \item[Case $\{gg\}$] We must have $\{c_1, c_2\} \subseteq \{\green, \blue\}$ and $c_3\in \{\green, \red\}$. Furthermore, the color of $c_2$ determines the colors of $c_1$ and $c_3$. So, we have $2$ colorings in this case.
    \item[Total] These add up to a total of $10+9+2+2+2 = 25$ extensions.
  \end{description}

\end{description}

  As in all cases we had at most $31$ extensions, the lemma is proved.
\end{proof}

Lemma~\ref{lemma:case_6} provides the number of extensions of colorings isomorphic to the one in~\ref{fig:2coloringK6}.

\begin{lemma}\label{lemma:case_6}
  Let $\phi$ be a Gallai colouring of $K_6$ with exactly three colors that contains a matching of size three colored with exactly two colors and all the remaining edges are colored with the third color.
  Then, $w(\phi) = 53$.
\end{lemma}

\begin{proof}
  Let $\phi$ be as in the statement.
  Let $b_1b_2$, $b_3b_4$, and $r_1r_2$ be the edges of the matching and suppose without loss of generality that $\phi(b_1b_2) = \phi(b_3b_4) = \blue$ and $\phi(r_1r_2) = \red$.
  We will count the number of ways to extend $\phi$ to a Gallai coloring of $K_7$ when we add a vertex $u$ to the initial $K_6$.

Note that all edges between $\{b_1, b_2\}$ and $V(K_6)\setminus \{b_1, b_2\}$ are red.
In our proof we consider separately some of the (seven) ways to color the edges $ub_1$ and $ub_2$ avoiding rainbow triangles.

  \textbf{Case $(\phi(ub_1), \phi(ub_2)) = (\red,\red)$:} In this case we do not have any additional restrictions on the choices of the colors of the remaining edges. Also, note that the coloring of the remaining graph is edge-special with three colors. Thus, we have $2^4 + 3$ ways to extend the coloring.

  \textbf{Cases $(\phi(ub_1), \phi(ub_2)) = (\green, \green)$:} We cannot use color blue on the remaining edges. So, we can only use red of green. Because $b_3b_4$ is blue, the edges $ub_3$ and $ub_4$ must have the same color. Other than that, there is no other restriction, so we have $2^3$ ways to color those remaining edges.

  \textbf{Cases $(\phi(ub_1), \phi(ub_2))\in \{(\blue,\blue),(\red,\blue),(\blue,\red)\}$:} Analogously to the previous case, all  remaining edges must be red of blue. Additionally, because $g_1g_2$ is green, $ug_1$ and $ug_2$ must have the same color. Thus, there are $2^3$ extensions in each of the three cases.

  \textbf{Cases $(\phi(ub_1), \phi(ub_2))\in\{(\blue, \green),(\green, \blue)\}$:} These cases allow us to use only the red color on the remaining edges.
\vspace{0.2cm}

  In total, there are $w(\phi) = 19 + 8 + 3\cdot 8 + 2= 53$ extensions.
\end{proof}

\end{document}